\documentclass[12 pt]{article}
\usepackage{stmaryrd}
\usepackage{times}
\usepackage{booktabs}
\usepackage{subfigure}
\usepackage{longtable}
\usepackage{setspace}
\usepackage{tabularx}
\usepackage{multirow, makecell}
\usepackage{rotating}
\usepackage{pifont}
\usepackage{floatrow}
\usepackage{mathrsfs}
\usepackage[fleqn]{amsmath}
\usepackage{amsfonts,amsthm,amssymb,mathrsfs,bbding}
\usepackage{txfonts}
\usepackage{graphics,multicol}
\usepackage{graphicx}
\usepackage{color}
\usepackage{caption}
\usepackage{indentfirst}
\usepackage{cite}
\usepackage{latexsym,bm}
\usepackage{enumerate}
\usepackage{epsfig}
\evensidemargin=\oddsidemargin\topmargin=-1.5cm

\theoremstyle{definition}
\newtheorem{definition}{Definition}[section]
\newtheorem{remark}{Remark}[section]
\newtheorem{example}{Example}[section]
\newtheorem{lemma}{Lemma}[section]
\newtheorem{theorem}{Theorem}[section]

\newtheorem{proposition}[theorem]{Proposition}
\newtheorem{corollary}{Corollary}[section]
\newtheorem{claim}{Claim}[section]
\newtheorem{problem}{Problem}[section]

\addtocounter{section}{0}
\begin{document}

\title{A relation on trees and the topological indices based on subgraph\footnote{This work is supported
by NSFC Grant No. 11971274, 12061074, 11671344.}}
\author{
{\small Rui Song$^{a}$,\ \ Qiongxiang Huang$^{a,}$\footnote{Corresponding author.
\newline{\it \hspace*{5mm}Email addresses:} huangqx@xju.edu.cn(Q.X. Huang).}\ \  }\\[2mm]
\footnotesize $^a$College of Mathematics and Systems Science, Xinjiang University, Urumqi, Xinjiang 830046, China\\
}
\date{}
\maketitle

{\flushleft\large\bf Abstract:}
A topological index reflects the physical, chemical and structural properties of a molecule, and its study has an important role in molecular topology, chemical graph theory and mathematical chemistry. It is a natural problem to characterize non-isomorphic graphs with the same  topological index value. By introducing  a relation on trees with respect to edge division vectors, denoted by $\langle\mathcal{T}_n, \preceq \rangle$, in this paper we give some results for  the relation order in $\langle\mathcal{T}_n, \preceq \rangle$, it allows us to compare the size of the topological index value without relying on the specific forms of them, and naturally we can determine which trees have the same topological index value. Based on these results we  characterize some classes of trees that are uniquely determined by their edge division vectors and  construct infinite classes of non-isomorphic trees with the same  topological index value, particularly such trees of order no more than $10$ are completely determined.

\begin{flushleft}

\end{flushleft}
\section{Introduction}
A topological index is usually defined by some graph invariants,  such as the number of vertices, the number of edges, vertex degree, degree sequence, matching number, etc \cite{Todeschini}. In the fields of bioinformatics, molecular topology, chemical graph theory and mathematical chemistry, a topological index is a type of molecular descriptors that are calculated based on the molecular graph of a chemical compound.

In the field of chemistry, it is common to calculate topological indices of the molecular graph in order to figure out the physical or chemical properties of a molecule statistically.
The inverse problem of topological indices is proposed by X. Li et al. in \cite{Li} as follows: given an index value, one wants to design chemical compounds (given as graphs or trees) having that index value, it is not necessary  to obtain all the isomorphic graphs. For more results on this topic, readers may  refer to \cite{Lang}. As the research developed, some researchers  tried to use some kind of topological index or a few topological indices for classification based on isomorphism \cite{Dehmer1,Dehmer2,Dehmer3}.
The main problem of classification based on isomorphism is that the topological indices may be identical even for two or several non-isomorphic graphs and  the situation becomes worse with the increment of vertices of graph \cite{Dehmer4}. There are some similar inverse problem such as finding and constructing of cospectral graphs \cite{Schwenk} and equienergetic graphs \cite{Brankov,Bonif,Indulal,Indulal-0,Tura}.

In \cite{Guo}, X. Guo and M. Randi\'{c}  characterized some classes of  trees with the same $JJ$ index. Recently, D. Vuki\v{c}evi\'{c} and J. Sedlar in \cite{Vuki} introduced a relation order on trees with respect to edge division vector. They also gave the relationship between  edge division vector and the topological index on the class of trees, which enables us to simply calculate the topological indices of  trees by their edge division vectors. In \cite{Song}, Song and Huang et al. gave a new criterion to determine the order of trees with respect to the edge division vectors. Based on these results a large of extremal trees are determined including old and new, one can refer to Table 3 in \cite{Song}.

In this paper, we focus on considering the problem of characterizing the non-isomorphic trees with the same topological index value without depending on the individual form of topological index.
In Section 2, we introduce some notions and symbols and give two lemmas to determine a kind of relation order on trees by using edge division vector. Further, we obtain tree conditions for determining the increase and decrease of  topological indices.
In Section 3, we give a graph transformation to construct the trees that have the same edge division vector. Moreover  we find some sufficient conditions to determine wether such trees are isomorphic or not. Based on these conditions, we can simply produce infinite families of non-isomorphic  trees that have the same edge division vector.
In Section 4, we give several classes of trees which are uniquely determined by edge division vectors.
In Section 5, we characterize some classes of  non-isomorphic trees with the same edge division vector. In particular,  all the pairs of non-isomorphic  trees of order no more than $10$ with the same edge division vector are classified.
Based on above results, in Section 6, without relying on the specific form of individual topological index we give some classes of non-isomorphic trees with the same topological index value.

\section{Preliminaries}
Let $G=(V(G),E(G))$ be a simple connected graph with vertex set $V(G)$ and edge set $E(G)$. For a vertex $v\in V(G)$, let $N_G(v)$ denote the set of neighbors of $v$ and the degree $d_G(v)=|N_G(v)|$. For a pair  of vertices $u,v\in V(G)$, denote by $d_G(u,v)$ the length of the shortest path connecting vertices $u$ and $v$. Usually we will write only $d(v), N(v)$ and $d(u,v)$ when it does not lead to confusion.

A connected graph is a tree, denoted by $T$, if it has no cycle. Let $S_n$ and $P_n$ denote the star and the path with order $n$, respectively. A vertex $v$ is a pendent vertex or leaf if $d_T(v)=1$ and  a branching vertex if $d_T(v)\geq 3$. The tree $T-v$ is defined by removing the vertex $v$ and deleting all edges incident to $v$ from $T$, and the tree $T-e$ is defined by removing the edge $e$ of $T$.
For a tree $T$ of order $n$, let $e=uv\in E(T)$. By  $T_u(e)$ and $T_v(e)$ we will denote the two components of $T-e$ containing $u$ and $v$, respectively. We denote $n_u(e)=|T_u(e)|$ and $n_v(e)=|T_v(e)|$. Furthermore,  we define an edge function $\mu_T(e)=\min\{n_u(e),n_v(e)\}$ and simply write as $\mu(e)$ if without confusion. By definition  we have $n_u(e)+n_v(e)=n$ and $\mu(e)\leq \lfloor\frac{n}{2}\rfloor$.

Let $\mathcal{T}_n$ denote the set of  trees on $n$ vertices. For a tree $T\in \mathcal{T}_n$, let $r_i(T)$ denote the number of edges in $T$ for which $\mu(e)=i$, i.e., $r_i(T)=|\{e\in E(T)\mid \mu(e)=i\}|$. It is clear that $r_1(T)$ is just the number of pendent edges and  $r_i(T)=0$ for every $i>\lfloor\frac{n}{2}\rfloor$ due to  $\mu(e)\leq \lfloor\frac{n}{2}\rfloor$. The edge division vector $\mathbf{r}(T)$ is defined as a vector
$\mathbf{r}(T)=(r_1(T),r_2(T),\ldots,r_{\lfloor\frac{n}{2}\rfloor}(T))$. We will write  $\mathbf{r}$ and $r_i$ when it does not lead to confusion.
Recently, D. Vuki\v{c}evi\'{c} and J. Sedlar in \cite{Vuki}  defined the order of edge division vectors:  two edge division vectors $\mathbf{r}$ and $\mathbf{r}'$ of trees $T$ and $T'$ in $\mathcal{T}_n$, respectively, have a relation, denoted by
$\mathbf{r}\preceq \mathbf{r}'$, if the inequality $\sum\limits_{i=k}^{\lfloor\frac{n}{2}\rfloor}r_i\leq \sum\limits_{i=k}^{\lfloor\frac{n}{2}\rfloor}r'_i$
holds for every $k=1,2,\ldots,\lfloor\frac{n}{2}\rfloor$. If the inequality is strict for at least one $k$, then we say that $\mathbf{r}\prec \mathbf{r}'$. Naturally, we define  $T\preceq T'$ if $\mathbf{r}\preceq \mathbf{r}'$ ($T\prec T'$ if $\mathbf{r}\prec \mathbf{r}'$). Specially, we denote by $T\approx T'$ if $\mathbf{r}=\mathbf{r'}$. Thus the trees of $\mathcal{T}_n$ is a set defined with the relation ``$\preceq $", which is denoted by $\langle\mathcal{T}_n, \preceq \rangle$. Two trees $T, T'$ are called EDV-equivalent trees if $T\approx T'$. However, $T\approx T'$ does not imply $T\cong T'$ (to see Figure \ref{fig-3} for example). A tree $T$ is said to be determined by its edge division vector $\mathbf{r}(T)$ (DEDV for short) if,
for any $T'\in \mathcal{T}_n$, we have $T'\cong T$ whenever $T\approx T'$.

First we give two lemmas to determine  the order on $\langle\mathcal{T}_n, \preceq \rangle$, from which we will characterize
the EDV-equivalent trees and DEDV-trees.

For $T,T'\in \mathcal{T}_n$, let $\varphi:E(T)\longrightarrow E(T')$ be a bijection. $T$ and $T'$ are said to be $(\varphi,\mu)$-similar with respect to $e_1\in E(T)$ if   $\mu_{T}(e)=\mu_{T'}(\varphi(e))$ for any $e\not=e_1$. We start with a lemma which is given in \cite{Song}.

\begin{lemma}[\cite{Song}]\label{lem-5}
Suppose that  $T, T'\in \mathcal{T}_n$  are $(\varphi,\mu)$-similar with respect to $e_1$, and $\varphi(e_1)=e_1'$. We have\\
(1) If $\mu_{T}(e_1)<\mu_{T'}(e_1')$, then $T\prec T'$; \\
(2) If $\mu_{T}(e_1)>\mu_{T'}(e_1')$, then $T\succ T'$;\\
(3) If $\mu_{T}(e_1)=\mu_{T'}(e_1')$, then $T\approx T'$.
\end{lemma}

For $T\in \mathcal{T}_n$, let $uv$ and $xu$ be edges of $T$. Denote by $T_x(ux)$  the component in $T-ux$ containing vertex $x$. Note that $T^*=T_u(ux)$. Let $T'=T-ux+xv$ be the tree obtained from $T$ by moving the component $T_x(ux)$ from $u$ to $v$ (see Figure \ref{fig-0}), such tree $T'$ we call a branch-moving of $T_x(ux)$ from $T$.

\begin{figure}[htbp]
  \centering
\unitlength 1mm 
\linethickness{0.4pt}
\ifx\plotpoint\undefined\newsavebox{\plotpoint}\fi 
\begin{picture}(66.242,19.269)(0,0)
\put(12.365,15.404){\circle*{1}}
\put(53.49,15.279){\circle*{1}}
\put(2,9.779){\scriptsize$T^*$}
\put(9,14.904){\scriptsize$u$}
\put(43,10.279){\scriptsize$T^*$}
\put(51,14.904){\scriptsize$u$}
\put(28.365,10.279){\vector(1,0){8.875}}
\put(12.365,15.154){\line(1,0){5.875}}
\put(21.741,13.404){\oval(6.75,9.25)[]}
\put(18.24,15.279){\circle*{1}}
\put(19.241,14.904){\scriptsize$x$}
\put(19.241,11.404){\tiny $T_x(ux)$}
\put(62.867,13.654){\oval(6.75,9.25)[]}
\put(59.367,15.529){\circle*{1}}
\put(60.367,15.154){\scriptsize$x$}
\put(60.367,11.654){\tiny $T_x(ux)$}
\multiput(53.421,15.209)(.9585,0){7}{{\rule{.4pt}{.4pt}}}
\put(12.521,10.576){\circle*{1}}
\put(9,9.899){\scriptsize$v$}
\put(7.955,11.844){\oval(15.91,14.849)[]}
\put(12.198,15.38){\line(0,-1){5.127}}
\put(53.887,10.576){\circle*{1}}
\put(51,9.899){\scriptsize$v$}
\put(49.321,11.844){\oval(15.91,14.849)[]}
\put(53.564,15.379){\line(0,-1){5.127}}
\multiput(53.917,10.253)(.037216146,.033727133){152}{\line(1,0){.037216146}}
\put(12.021,0){\scriptsize$T$}
\put(47,0){\scriptsize$T'=T-ux+vx$}
\end{picture}
  \caption{The branch-moving transformation}\label{fig-0}
\end{figure}
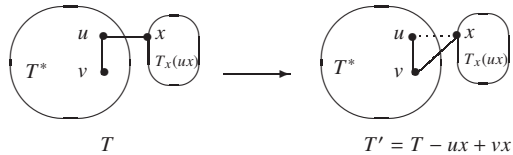

\begin{lemma}\label{lem-2-2}
For  $T\in \mathcal{T}_n$, let $uv$ and $xu$ be edges of $T$. Let $T'=T-ux+xv$. We have\\
(1) If $n_u(uv)\leq n_v(uv)$  then $T\succ T'$;\\
(2) If $n_u(uv)> n_v(uv)$  then
\[\left\{\begin{array}{ll}
T\prec T' & \mbox{ whenever $n_u(uv)-n_v(uv)>|T_x(ux)|$},\\
T\succ T' & \mbox{ whenever $n_u(uv)-n_v(uv)<|T_x(ux)|$},\\
T\approx T' & \mbox{ whenever $n_u(uv)-n_v(uv)=|T_x(ux)|$}.\\
\end{array}\right.
\]
\end{lemma}
\begin{proof}
First  we define bijection $\varphi: E(T)\longrightarrow E(T')$ such that
\begin{equation}\label{simi-eq-2}
\varphi(e)=\left\{\begin{array}{ll}
e & \mbox{ if $e\neq ux$ is an edge of $T$,}\\ \nonumber
vx & \mbox{ if $e= ux$.  }
\end{array}\right.
\end{equation}
It is clear that $\mu_{T}(e)=\mu_{T'}(\varphi(e))$ if $e\not=uv$. Therefore, $T$ and $T'$ are $(\varphi,\mu)$-similar with respect to $e_1=uv$.

If $n_u(uv)\leq n_v(uv)$, then $\mu_{T}(e_1)=\min\{n_u(uv),n_v(uv)\}=n_u(uv)$. We have
\[\mu_{T}(e_1)=n_u(uv)>n_u(uv)-|T_x(ux)|=n_u'(uv)=\mu_{T'}(uv)=\mu_{T'}(e_1).\]
Therefore, we have $T\succ T'$ by Lemma \ref{lem-5} (2), and (1) holds.

If $n_u(uv)> n_v(uv)$, then $\mu_{T}(e_1)=\min\{n_u(uv),n_v(uv)\}=n_v(uv)$. Notice that $\mu_{T'}(e_1)=\min\{n_u'(uv),n_v'(uv)\}$ and
\begin{equation}\label{nn-eq-1}\left\{\begin{array}{ll}
n_u'(uv)=n_u(uv)-|T_x(ux)|,\\ \nonumber
n_v'(uv)=n_v(uv)+|T_x(ux)|.
\end{array}\right.
\end{equation} It is easy to verify that
$$\left\{\begin{array}{ll}
\mu_{T}(e_1)=n_v(uv)<n_v(uv)+|T_x(ux)|=\mu_{T'}(e_1) & \mbox{ if $n_u(uv)-n_v(uv)\ge 2|T_x(ux)|$,}\\
\mu_{T}(e_1)=n_v(uv)> n_u(uv)-|T_x(ux)|=\mu_{T'}(e_1) & \mbox{ if $n_u(uv)-n_v(uv)<|T_x(ux)|$,}\\
\mu_{T}(e_1)=n_v(uv)= n_u(uv)-|T_x(ux)|=\mu_{T'}(e_1) & \mbox{ if $n_u(uv)-n_v(uv)=|T_x(ux)|$,}\\
\mu_{T}(e_1)=n_v(uv)<  n_u(uv)-|T_x(ux)|=\mu_{T'}(e_1) & \mbox{ if  $|T_x(ux)|<n_u(uv)-n_v(uv)< 2|T_x(ux)|$,}\\
\end{array}\right.
$$
\normalsize that is,
\[\left\{\begin{array}{ll}
\mu_{T}(e_1)<\mu_{T'}(e_1) & \mbox{ if $n_u(uv)-n_v(uv)>|T_x(ux)|$,}\\
\mu_{T}(e_1)> \mu_{T'}(e_1) & \mbox{ if $n_u(uv)-n_v(uv)<|T_x(ux)|$,}\\
\mu_{T}(e_1)= \mu_{T'}(e_1) & \mbox{ if $n_u(uv)-n_v(uv)=|T_x(ux)|$.}\\
\end{array}\right.\]
It follows (2) by Lemma \ref{lem-5}.

We complete this proof.
\end{proof}
Lemma \ref{lem-2-2} indicates the changes for the order $\preceq$ in $\langle\mathcal{T}_n, \preceq \rangle$ when we apply branch-moving along an edge, which can be used to find the maximum or minimum tree under the meaning of order $\preceq$.

\section{Branch-exchange for trees}

In this section we will introduce a graph transformation, called the branch-exchange, which will be used to construct the pairs of EDV-equivalent trees or DEDV-trees.

For a tree  $T\in \mathcal{T}_n$, let $u$ and $v$ be two  vertices of $T$ and $P_{uv}=uu_2\cdots u_{k-1}v$ be the path connecting $u$ and $v$.  Now let
$$T_{u}(P_{uv})=\{T_x(ux)\mid x\in N_T(u)\setminus u_2\},$$
which is called the $u$-branch of $T$ (with respect to $P_{uv}$). Each $T_x(ux)\in T_{u}(P_{uv})$ is a subtree  in $T-u$ that has $x\in N_T(u)\setminus u_2$ as its root vertex, and $v$-branch  $T_{v}(P_{uv})=\{T_y(vy)\mid y\in N_T(v)\setminus u_{k-1}\}$ is similarly defined.
We say that two subsets $S(u)=\{x_1,\ldots,x_s\}\subseteq N_T(u)\setminus u_2$ and  $S(v)=\{y_1,\ldots,y_t\}\subseteq N_T(v)\setminus u_{k-1}$ are \emph{balanced} if $$\sum_{x_i\in S(u)}|T_{x_i}(ux_i)|=\sum_{y_j\in S(v)}|T_{y_j}(vy_j)|.$$
Further, we call $T_{S(u)}=\{T_{x_i}(ux_i)\mid x_i\in S(u)\}$ and $T_{S(v)}=\{T_{y_j}(vy_j)\mid y_j\in S(v)\}$ the \emph{balanced components} with respect to $u$ and $v$. By deleting the balanced components from $T$, we get
$$T^*=T-(\sum_{x_i\in S(u)}ux_i+T_{S(u)})-(\sum_{y_j\in S(v)}vy_j+T_{S(v)}).$$
Hence, we have  $T^*_{u}(P_{uv})=T_{u}(P_{uv})-T_{S(u)}$ and $T^*_{v}(P_{uv})=T_{v}(P_{uv})-T_{S(v)}$ (see Figure \ref{fig-2}).

\begin{definition}\label{def-3-2}
Let $T'$  be the tree obtained from $T$ by exchanging the  balanced components $T_{S(u)}$ and $T_{S(v)}$ (see Figure \ref{fig-2}), that is
\begin{equation}\label{T1-T2-eq-1} \left\{\begin{array}{ll}
T=T^*+(\sum_{x_i\in S(u)}ux_i+T_{S(u)})+(\sum_{y_j\in S(v)}vy_j+T_{S(v)})\\
T'=T^*+(\sum_{y_j\in S(v)}uy_j+T_{S(v)})+(\sum_{x_i\in S(v)}vx_i+T_{S(u)}).
\end{array}\right.
\end{equation}
We call $T'$ the \emph{branch-exchange} of $T$  with $T_{S(u)}$ and $T_{S(v)}$.
\end{definition}

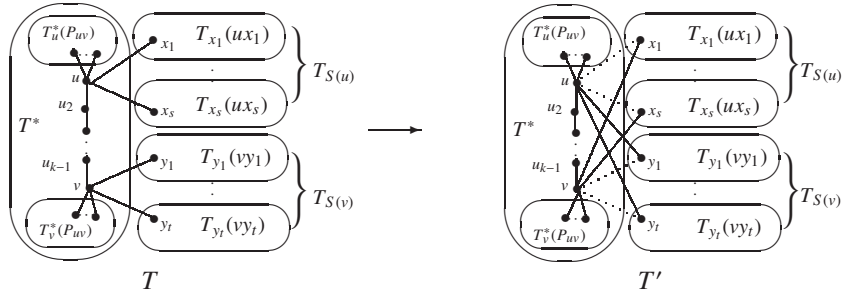
\begin{figure}[htb]
\centering
\unitlength 1.2mm 
\linethickness{0.4pt}
\ifx\plotpoint\undefined\newsavebox{\plotpoint}\fi 
\begin{picture}(87.5,31.642)(0,0)
\multiput(8.809,22.75)(.049144828,.03362069){145}{\line(1,0){.049144828}}
\multiput(8.809,22.875)(.074742268,-.033505155){97}{\line(1,0){.074742268}}
\multiput(8.809,11.375)(.076623656,.033602151){93}{\line(1,0){.076623656}}
\multiput(8.809,11)(.076623656,-.033602151){93}{\line(1,0){.076623656}}
\put(15.809,27.625){\circle*{.8}}
\put(15.935,19.625){\circle*{.8}}
\put(8.809,11.25){\circle*{.8}}
\put(15.935,14.5){\circle*{.8}}
\put(15.935,7.75){\circle*{.8}}
\put(.849,17.375){\scriptsize$T^*$}
\put(6.8,23){\tiny$u$}
\put(7,11){\tiny$v$}
\put(16.649,26.75){\tiny$x_1$}
\put(22.149,23){\tiny$\vdots$}
\put(16.649,19.325){\tiny$x_s$}
\put(16.649,13.625){\tiny$y_1$}
\put(22.149,9.3){\tiny$\vdots$}
\put(16.649,6.75){\tiny$y_t$}
\put(14.5,0){\footnotesize$T$}
\put(20.5,27.572){\scriptsize$T_{x_1}(ux_1)$}
\put(20.5,19.8){\scriptsize$T_{x_s}(ux_s)$}
\put(31,24){\Bigg\}}
\put(33.5,23.5){\scriptsize$T_{S(u)}$}
\put(21,13.96){\scriptsize$T_{y_1}(vy_1)$}
\put(21,6.4){\scriptsize$T_{y_t}(vy_t)$}
\put(31,10){\Bigg\}}
\put(33.5,9.5){\scriptsize$T_{S(v)}$}
\put(22.188,20.589){\oval(16.794,5.127)[]}
\put(22.011,28.015){\oval(16.794,5.127)[]}
\put(22.365,14.58){\oval(16.794,5.127)[]}
\put(22.542,6.801){\oval(16.794,5.127)[]}
\put(62.782,22.875){\circle*{.8}}
\put(62.782,11.125){\circle*{.8}}
\put(55.649,17){\scriptsize$T^*$}
\put(61,22.625){\tiny$u$}
\put(61,11){\tiny$v$}
\put(69.5,0){\footnotesize$T'$}
\put(69.783,27.625){\circle*{.8}}
\put(69.908,19.625){\circle*{.8}}
\put(69.908,14.5){\circle*{.8}}
\put(69.908,7.75){\circle*{.8}}
\put(70.649,26.75){\tiny$x_1$}
\put(75.649,23){\tiny$\vdots$}
\put(70.649,19.325){\tiny$x_s$}
\put(70.649,13.625){\tiny$y_1$}
\put(75.649,9.3){\tiny$\vdots$}
\put(70.649,6.75){\tiny$y_t$}
\multiput(69.783,27.625)(-.03372093,-.077906977){215}{\line(0,-1){.077906977}}
\multiput(69.908,19.5)(-.03372093,-.040697674){215}{\line(0,-1){.040697674}}
\multiput(69.783,14.625)(-.03372093,.038372093){215}{\line(0,1){.038372093}}
\multiput(62.533,22.875)(.033675799,-.069634703){219}{\line(0,-1){.069634703}}
\multiput(69.588,27.805)(-.7125,-.5){11}{{\rule{.4pt}{.4pt}}}
\multiput(62.463,22.805)(.80556,-.375){10}{{\rule{.4pt}{.4pt}}}
\multiput(69.713,14.43)(-.77778,-.38889){10}{{\rule{.4pt}{.4pt}}}
\multiput(62.713,10.93)(.79167,-.36111){10}{{\rule{.4pt}{.4pt}}}
\put(75,27.572){\scriptsize$T_{x_1}(ux_1)$}
\put(75,19.63){\scriptsize$T_{x_s}(ux_s)$}
\put(85.5,24){\Bigg\}}
\put(87.5,23.5){\scriptsize$T_{S(u)}$}
\put(76,14){\scriptsize$T_{y_1}(vy_1)$}
\put(76,6.2){\scriptsize$T_{y_t}(vy_t)$}
\put(85.5,10){\Bigg\}}
\put(87.5,9.7){\scriptsize$T_{S(v)}$}
\put(76.636,20.587){\oval(16.794,5.127)[]}
\put(76.459,28.013){\oval(16.794,5.127)[]}
\put(76.813,14.578){\oval(16.794,5.127)[]}
\put(76.989,6.801){\oval(16.794,5.127)[]}
\put(8.419,22.919){\line(0,-1){3.122}}
\put(8.419,19.798){\line(0,-1){1.932}}
\put(8.419,14.743){\line(0,-1){3.568}}
\multiput(8.568,22.919)(-.03336735,.07583673){49}{\line(0,1){.07583673}}
\multiput(8.419,23.068)(.03302778,.09083333){36}{\line(0,1){.09083333}}
\multiput(8.716,11.176)(-.03302222,-.07268889){45}{\line(0,-1){.07268889}}
\multiput(8.716,11.324)(.033037,-.1266296){27}{\line(0,-1){.1266296}}
\put(7.081,26.19){\circle*{.8}}
\put(9.608,26.041){\circle*{.8}}
\put(7.23,8.203){\circle*{.8}}
\put(9.46,8.054){\circle*{.8}}
\put(7.5,25.5){\tiny$\cdots$}
\put(7.5,7.8){\tiny$\cdots$}
\put(3.5,28){\tiny$T^*_u(P_{uv})$}
\put(3.2,6){\tiny$T^*_v(P_{uv})$}
\put(8.27,19.946){\circle*{.8}}
\put(8.419,17.568){\circle*{.8}}
\put(8.419,14.297){\circle*{.8}}
\put(5.3,20.095){\tiny$u_2$}
\put(8.2,15){\tiny$\vdots$}
\put(3.5,13.5){\tiny$u_{k-1}$}
\put(62.677,22.622){\line(0,-1){3.122}}
\put(62.677,14.446){\line(0,-1){3.568}}
\multiput(62.826,22.622)(-.03336735,.07583673){49}{\line(0,1){.07583673}}
\multiput(62.677,22.771)(.03302778,.09083333){36}{\line(0,1){.09083333}}
\multiput(62.974,10.878)(-.03302222,-.07266667){45}{\line(0,-1){.07266667}}
\multiput(62.974,11.027)(.033037,-.1266296){27}{\line(0,-1){.1266296}}
\put(61.339,25.892){\circle*{.8}}
\put(63.866,25.744){\circle*{.8}}
\put(61.488,7.905){\circle*{.8}}
\put(63.718,7.757){\circle*{.8}}
\put(61.7,25.5){\tiny$\cdots$}
\put(61.7,7.3){\tiny$\cdots$}
\put(57.9,28){\tiny$T^*_u(P_{uv})$}
\put(57.9,5.5){\tiny$T^*_v(P_{uv})$}
\put(62.528,19.649){\circle*{.8}}
\put(62.677,17.27){\circle*{.8}}
\put(62.677,14){\circle*{.8}}
\put(59.5,19.798){\tiny$u_2$}
\put(62.4,14.8){\tiny$\vdots$}
\put(57.8,13.3){\tiny$u_{k-1}$}
\put(62.528,19.5){\line(0,-1){1.932}}
\put(6.875,27.563){\oval(9.75,5.375)[]}
\put(6.615,17.446){\oval(13.23,28.392)[]}
\put(6.625,7.313){\oval(9.75,5.375)[]}
\put(8.4,23.15){\circle*{.8}}
\put(61.625,27.438){\oval(9.75,5.375)[]}
\put(61.365,17.321){\oval(13.23,28.392)[]}
\put(61.375,7.188){\oval(9.75,5.375)[]}
\put(39.625,17.625){\vector(1,0){5.875}}
\end{picture}
  \caption{The branch-exchange transformation}\label{fig-2}
\end{figure}

It is worth mentioning that $u$ and $v$ do not need to be adjacent vertices in $T$ for the branch-exchange. Additionally, each component $T_{x_i}(ux_i)\in T_{S(u)}$  is a subtree of $T$ with root vertex $x_i\in S(u)\subset N(u)$ (see Figure \ref{fig-2}).

\begin{lemma}\label{lem-2-3}
Suppose that $T\in \mathcal{T}_n$ has two balanced  components $T_{S(u)}$ and $T_{S(v)}$  and let $T'$ be the branch-exchange of $T$ with $T_{S(u)}$ and $T_{S(v)}$  shown in  Figure \ref{fig-2}. We have $T\approx T'$.
\end{lemma}
\begin{proof}
Let $S(u)=\{x_1,\ldots,x_s\}\subseteq N(u)\setminus \{v\}$ and  $S(v)=\{y_1,\ldots,y_t\}\subseteq N(v)\setminus \{u\}$ be two balanced  subsets of $V(T)$. First  we define bijection $\varphi: E(T)\longrightarrow E(T')$ such that
\begin{equation}\label{simi-eq-2}\varphi(e)=\left\{\begin{array}{ll}
e & \mbox{ if $e\neq ux_i, vy_j$ for $i=1,\ldots,s$, $j=1,\ldots,t$,}\\ \nonumber
vx_i & \mbox{ if $e= ux_i$ for $i=1,\ldots,s$,}\\ \nonumber
uy_j & \mbox{ if $e= vy_j$ for $j=1,\ldots,t$.}
\end{array}\right.
\end{equation}
It is routine to verify   that $\mu_{T}(e)=\mu_{T'}(\varphi(e))$ for any $e\in E(T)$ since $S(u)$ and $S(v)$ are balanced.
It follows $T\approx T'$ by Lemma \ref{lem-5}(3).
\end{proof}

For a tree $F$ with vertex $x$, let $F_x$ denote the  tree with specified root vertex $x$. For two trees $F$ and $H$,  $F_x$ and $H_y$ are said to be \emph{strongly isomorphic}, denoted by $F_x\simeq H_y$,  if there exists an isomorphism $\varphi$ from $F$ to $H$ such that $\varphi(x)=y$. It is clear that $F_x\simeq H_y$ implies $F_x\cong H_y$ but not vice versa. Let $F^f=\cup_{i=1}^sF_{x_i}$ and $H^f=\cup_{i=1}^tH_{y_i}$, where $F_{x_1},\ldots,F_{x_s}$ (resp., $H_{y_1},\ldots,H_{y_t}$) are vertex disjoint trees. We say that $F^f\simeq H^f$ if $s=t$ and
there exists a permutation $\phi$ on $\{1,2,\ldots,s\}$ such that $F_{x_i}\simeq H_{y_{\phi(i)}}$ for $i=1,\ldots,s$.

Let $\alpha$ be an automorphism of $T$, i.e., $\alpha\in Aut(T)$.
The image of $P_{uv}$ under $\alpha$ is also a path, say $\alpha(P_{uv})=u'u_2'\cdots u_{k-1}'v'=P_{u'v'}$ where $\alpha(u_i)=u_i'$ for $i=2,\ldots,k-1$ and $\alpha(u)=u'$, $\alpha(v)=v'$. By $T_{u}(P_{uv})\simeq T_{\alpha(u)}(\alpha(P_{uv}))= T_{u'}(P_{u'v'})$ we mean that, for each  $T_x(xu)\in T_{u}(P_{uv})$, there is some $T_z(zu')\in T_{u'}(P_{u'v'})$ such that $T_x(xu)\simeq T_z(zu')$ (i.e., $T_x(xu)$ and $ T_z(zu')$ are strongly isomorphic). We call $u$ and $v$  \emph{similar} if there exists  $\alpha\in Aut(T)$ such that $\alpha$ contains the transposition $(u\ v)$ (i.e., $\alpha(u)=v$ and $\alpha(v)=u$). Obviously, if  $u$ and $v$  are similar   then $T_{u}(P_{uv})\simeq T_{v}(P_{uv})$.

By using the above  notions and symbols we can prove the following result.

\begin{lemma}\label{thm-3-2-0}
Let $T$ and $T'$ be the trees described in (\ref{T1-T2-eq-1}) and shown in Figure \ref{fig-2}, where $T_{S(u)}$ and $T_{S(v)}$ are  balanced components. Then $T\approx T'$. Moreover, if $T\cong T'$, then either $T_{S(u)}\simeq T_{S(v)}$ or $u$ and $v$ are similar in $T^*$.
\end{lemma}
\begin{proof}
Since $T'$  is a branch-exchange of $T$ with  balanced components $T_{S(u)}$ and $T_{S(v)}$, we have $T'\approx T$ by  Lemma \ref{lem-2-3}.

Let  $B_u=\{T_{\alpha(u)}(\alpha(P_{uv}))\mid \alpha\in Aut (T)\}$ be the set of strongly isomorphic copies of $u$-branch that consists of an orbit of $Aut(T)$, and similarly we define $B_v=\{T_{\alpha(v)}(\alpha(P_{uv}))\mid \alpha\in Aut (T)\}$. Note that the $u$-branch $T_{u}(P_{uv})\in B_u$, we have $b_u=|B_u|\ge 1$, and all the copies of the $u$-branch in $B_u$ are vertex-disjoint due to $T$ is a tree, similarly we have $b_v=|B_v|\ge 1$.

Since  $ T\cong T'$,  $T'$ also contains $b_u$ copies of  $u$-branch and $b_v$ copies of $v$-branch.   From Figure \ref{fig-2} and  the representation of $T$ in (\ref{T1-T2-eq-1}), we see that besides of $T_{u}(P_{uv})$ the other  $b_u-1$ number  of $u$-branches are included in $T^*$, similarly besides of $T_{v}(P_{uv})$ the other  $b_v-1$ number  of $v$-branches are also included in $T^*$. Therefore, $T^*$ contains exactly $(b_u-1)$'s $u$-branches in $B_u$ and  $(b_v-1)$'s $v$-branches in $B_v$. On the other hand, from Figure \ref{fig-2} we see that $T_{u}'(P_{uv})=T_u^*(P_{uv})\cup T_{S(v)}$ and  $T_{v}'(P_{uv})=T_v^*(P_{uv})\cup T_{S(u)}$ are only two branches of $T'$ not included in $T^*$, it implies that they must be one $u$-branch and one $v$-branch since  otherwise $T'$ will contain at most $b_u-1$ numbers of $u$-branches or $b_v-1$ numbers of $v$-branches. Hence $\{T_{u}(P_{uv}),T_{v}(P_{uv})\}=\{T_{u}'(P_{uv}),T_{v}'(P_{uv})\}$. If $B_u$ and $B_v$ are distinct orbits then $T_u^*(P_{uv})\cup T_{S(u)}=T_{u}(P_{uv})\simeq T_{u}'(P_{uv})=T_u^*(P_{uv})\cup T_{S(v)}$, and so $T_{S(u)}\simeq  T_{S(v)}$. If $B_u$ and $B_v$ are identified then $T_u^*(P_{uv})\cup T_{S(u)}=T_{u}(P_{uv})\simeq T_{v}'(P_{uv})=T_v^*(P_{uv})\cup T_{S(u)}$, and so  $T_u^*(P_{uv})\simeq T_v^*(P_{uv})$. Thus  there exists $\alpha\in Aut(T)$ such that $\alpha(u)=v$.  It implies that $T^*$ has  an automorphism $\alpha^*$ containing transposition $(u\ v)$, i.e., $u$ and $v$ are similar in $T^*$.

We complete the proof.
\end{proof}

 It immediately follows the following three results from Lemma \ref{thm-3-2-0}.

\begin{theorem}\label{thm-3-2-1}
Under the assumption of Lemma \ref{thm-3-2-0}, if   $T_{S(u)}\not\simeq T_{S(v)}$ and $u,v$ are not similar in $T^*$, then $T\approx T'$ but  $T\not\cong T'$.
\end{theorem}

If $T_u^*(P_{uv})\not\simeq T_v^*(P_{uv})$ then $u$ and $v$ are not similar in $T^*$. From Theorem \ref{thm-3-2-1} we have

\begin{corollary}\label{cor-3-2-0}
Under the assumption of Lemma \ref{thm-3-2-0}. If $T_{S(u)}\not\simeq T_{S(v)}$ and $T_u^*(P_{uv})\not\simeq T_v^*(P_{uv})$, then $T\approx T'$ but  $T\not\cong T'$.
\end{corollary}

\begin{corollary}\label{cor-3-2-2}
Under the assumption of Lemma \ref{thm-3-2-0}. If $Aut(T^*)$ is ordinary  and  $T_{S(u)}\not\simeq T_{S(v)}$, then $T\approx T'$ but  $T\not\cong T'$.
\end{corollary}

\begin{example}\label{TT-exa-1}
Let $T^*$ be the tree with $u,v\in V(T^*)$ as shown in Figure \ref{fig-exa}. Let $T_x=P_3$, $x$ be the centre vertex of $P_3$, and $T_y=P_3$, $y$ be the one endpoint of $P_3$. Note that $T_x\not\simeq T_y$. Construct the trees $T$  and $T'$ as shown in Figure \ref{fig-exa}. Recall that $T^*_u(P_{uv})=P_3\not\simeq P_1=T^*_v(P_{uv})$. By Corollary \ref{cor-3-2-0}, we have $T\approx T'$ but  $T\not\cong T'$, where $\mathbf{r}(T)=\mathbf{r}(T')=(7,1,3,0,1,0,1)$.
\end{example}

\begin{figure}[h]
  \centering
\unitlength 0.8mm 
\linethickness{0.4pt}
\ifx\plotpoint\undefined\newsavebox{\plotpoint}\fi 
\begin{picture}(102.169,39.682)(0,0)
\put(8.619,17.818){\line(1,0){25.858}}
\multiput(8.619,17.187)(-.03851145,-.033694656){131}{\line(-1,0){.03851145}}
\put(21.863,17.818){\line(0,-1){8.199}}
\put(3.574,12.983){\circle*{1.2}}
\put(8.829,17.187){\circle*{1.2}}
\put(15.346,17.818){\line(0,1){7.778}}
\put(15.346,26.017){\circle*{1.2}}
\put(15.557,17.608){\circle*{1.2}}
\put(22.074,17.608){\circle*{1.2}}
\put(22.074,9.619){\circle*{1.2}}
\put(28.17,17.818){\circle*{1.2}}
\put(34.477,17.608){\circle*{1.2}}
\multiput(8.619,17.187)(.03370229,-.035305344){131}{\line(0,-1){.035305344}}
\put(12.824,12.352){\circle*{1.2}}
\put(15,14.5){\tiny$u$}
\put(28,14.5){\tiny$v$}
\put(12,25.176){\tiny$x$}
\put(25,24){\tiny$y$}
\put(20.602,12.352){\oval(41.204,15.977)[]}
\multiput(11.562,32.534)(.03364,-.053824){125}{\line(0,-1){.053824}}
\put(15.767,25.806){\line(0,1){7.148}}
\put(11.352,31.903){\circle*{1.2}}
\put(15.977,32.323){\circle*{1.2}}
\put(13.454,30.642){\oval(9.67,16.818)[]}
\put(28.38,24.755){\circle*{1.2}}
\put(28.17,28.96){\circle*{1.2}}
\put(27.75,33.585){\circle*{1.2}}
\put(28.17,30.221){\oval(9.67,16.818)[]}
\put(28.17,20.341){\line(0,1){13.875}}
\put(12,35){\scriptsize$T_x$}
\put(27,35){\scriptsize$T_y$}
\put(27.96,8.358){\scriptsize$T^*$}
\put(20.602,0){\scriptsize$T$}
\put(69.584,18.449){\line(1,0){25.858}}
\multiput(69.584,17.818)(-.03851145,-.033694656){131}{\line(-1,0){.03851145}}
\put(82.828,18.449){\line(0,-1){8.199}}
\put(64.539,13.614){\circle*{1.2}}
\put(69.794,17.818){\circle*{1.2}}
\put(76.311,26.648){\circle*{1.2}}
\put(76.522,18.239){\circle*{1.2}}
\put(83.039,18.239){\circle*{1.25}}
\put(83.039,10.25){\circle*{1.2}}
\put(89.135,18.449){\circle*{1.2}}
\put(95.442,18.239){\circle*{1.2}}
\multiput(69.584,17.818)(.03370229,-.035305344){131}{\line(0,-1){.035305344}}
\put(73.789,12.983){\circle*{1.2}}
\put(76,15){\tiny$u$}
\put(89,15){\tiny$v$}
\put(73,25.807){\tiny$x$}
\put(86,25){\tiny$y$}
\put(81.567,12.983){\oval(41.204,15.977)[]}
\multiput(72.527,33.165)(.03364,-.053824){125}{\line(0,-1){.053824}}
\put(76.732,26.437){\line(0,1){7.148}}
\put(72.317,32.534){\circle*{1.2}}
\put(76.942,32.954){\circle*{1.2}}
\put(74.419,31.273){\oval(9.67,16.818)[]}
\put(89.345,25.386){\circle*{1.2}}
\put(89.135,29.591){\circle*{1.2}}
\put(88.715,34.216){\circle*{1.2}}
\put(89.135,30.852){\oval(9.67,16.818)[]}
\put(72,36){\scriptsize$T_x$}
\put(88,36){\scriptsize$T_y$}
\put(88.925,8.989){\scriptsize$T^*$}
\put(81.567,.3){\scriptsize$T'$}
\multiput(76.732,26.647)(.050452,-.033636){250}{\line(1,0){.050452}}
\multiput(88.925,34.005)(.0323077,-.7438462){13}{\line(0,-1){.7438462}}
\multiput(89.345,24.335)(-.073171271,-.033685083){181}{\line(-1,0){.073171271}}
\put(28.244,21.068){\line(0,-1){3.27}}
\end{picture}
  \caption{$T$ and $T'$}\label{fig-exa}
\end{figure}

\begin{remark}
As Example \ref{TT-exa-1}, by using Corollary \ref{cor-3-2-0} or Corollary \ref{cor-3-2-2}, one can simply construct infinite pairs of EDV-equivalent trees.
\end{remark}

\section{The DEDV trees}
In this section we completely determine the DEDV-starlike trees, and  some other DEDV-trees are also characterized.

Recall that a tree $T\in \mathcal{T}_n$ is a  DEDV-tree if $T'\cong  T$ whenever $T'\approx T$. Denote by $\mathcal{B}(T)$  the set of the trees that are constructed from $T$ by branch-exchange transformation. To exactly, $T'\in \mathcal{B}(T)$ if and only if there exist $T=T_1$, $T_2$,\ldots,$T_t=T'$ such that $T_{i+1}$ is a branch-exchange of $T_i$ for $i=1,\ldots,t-1$. By Lemma \ref{lem-2-3}, we have   $T'\approx T$ for $T'\in \mathcal{B}(T)$.  According to definition we have the following result.

\begin{lemma}\label{thm-3-3}
If  $T\in \mathcal{T}_n$ is a  DEDV-tree, then  $T'\cong T$ whenever $T'\in \mathcal{B}(T)$, i.e., $\mathcal{B}(T)=\{T\}$.
\end{lemma}

In the next of this section, we will characterize some classes of DEDV-tree.
For a tree $T\in \mathcal{T}_n$, let $e=uv\in E(T)$, recall that $T_u(e)$ and $T_v(e)$ be respectively the two components of $T-e$ containing root vertices $u$ and $v$, and $\mu_T(e)=\min\{|T_u(e)|,|T_v(e)|\}$.  For any $u\in V(T)$, there exists at most one edge $e$ incident to $u$ such that $|T_u(e)|\leq\lfloor\frac{n}{2}\rfloor$. We call $T_u(e)$ a pendent subtree of $T$ (with respect to root vertex $u$) if  $\mu_T(e)=|T_u(e)|$. We will write only $T_u$ when it does not lead to confusion. In particular, if $T_u$ is really a path and $d_T(u)=2$, it is called a pendent path; if it is a star with $u$ as its center vertex, it is called a pendent star.
A pendent subtree $T_u$ is maximal if $u$ is suspended from a branching vertex (or there is no any pendent subtree of order $|T_u|+1$ including $T_u$). Let $\mathbf{r}(T)=(r_1,r_2,\ldots,r_{\lfloor\frac{n}{2}\rfloor})$. Note that $r_i=|\{e\in E(T)\mid \mu_T(e)=i\}|$, thus $r_i$ is just the number of pendent subtrees of $T$ with order $|T_u|=i$. Particularly,  if $T_u$ is a pendent subtree with $|T_u|=1$ then  $u$ is a leaf of $T$, in this case $T_u$ contributes one to $r_1$; if $|T_u|=2$ then subtree $T_u$ is pendent path $P_2$, which contributes one to $r_2$. Naturally we have

\begin{claim}\label{r-claim-1}
For   $T,T'\in \mathcal{T}_n$ let $\mathbf{r}(T)=(r_1,r_2,\ldots,r_{\lfloor\frac{n}{2}\rfloor})=\mathbf{r}(T')$. We have \\
(a) $T$ and $T'$ have exactly $r_1$ pendent vertices;\\
(b) $T$  and $T'$ have exactly $r_2$ pendent $P_2$;\\
(c) $T$ and $T'$ have the same number of maximal pendent paths, the length of which is equal for the path with the smallest length.\\
(d) $r_1+r_2+\cdots+r_{\lfloor\frac{n}{2}\rfloor}=n-1=|E(T)|=|E(T')|$.
\end{claim}

Since the  pendent subtree of order three is either $P_3$ or $S_3$,  $r_3$ is the total number of the  pendent subtrees  $P_3$ and $S_3$. However, if $r_1=r_2=r_3=k$ then $T$ must contain exactly $k$ pendent paths $P_3$, since otherwise if $T$ has a pendent subtree  $S_3$ then $T$ has at least $k+1$ pendent vertices which contradicts $r_1=k$. In general, let $T_{n;k}$  be the starlike tree with branching vertex $u$ such that $T_{n;k}-u=k*P_s$(i.e., $k$ copies of $P_s$), where $ks+1=n$. We have the following result.

\begin{lemma}\label{r-lem-1}
Let $\mathbf{r}(T)=(r_1,r_2,\ldots,r_{\lfloor\frac{n}{2}\rfloor})$ be the edge division vector of $T\in \mathcal{T}_n$. Suppose that $r_1=\cdots=r_s=k$. Then $T$ has exactly $k$ pendent paths $P_s$ with respect to their endpoints. Moreover \\
(a)  If $r_i=0$ for $i>s$ then $T=T_{n;k}$.\\
(b)  If $r_{s+1}=t<k$ and $r_i=0$ for $i>s+1$ then $T$ is a starlike tree with a center vertex $u$ such that $T-u=(k-t)*P_s\cup t*P_{s+1}$, where $n-1=s k+t$.
\end{lemma}
\begin{proof}
By the arguments as $s=3$ as above, we first claim that $T$  has exactly $k$ pendent $P_s$ since otherwise $T$ has leaves more than $k$.

(a) Since $r_i=0$ for $i>s$, we have $r_1+\cdots+r_s=n-1$ by  Claim \ref{r-claim-1}(d). It implies that the endpoints of these $k$ pendent $P_s$ join at a center vertex $u$, that is $T=T_{n;k}$.

(b) Now  $T$  contains  exactly $k$ pendent $P_s$ each of them contributes  one leaf to $T$. Also note that $r_{s+1}=t$, we see that $T$ has $t$ pendent subtrees of order $s+1$. Let $T_v$ be such a pendent subtree with respect to root $v$, where $|T_v|=s+1$. If $d_{T_v}(v)\ge2$, let $x_1$ and $x_2$ be two adjacent vertices of $v$ in $T_v$, then $F=T_v$ has two subtree $F_{x_1}$ and $F_{x_2}$, where $F_{x_i}$ is the component in $F-x_iv$ containing $x_i$ for $i=1,2$. Clearly  $|F_{x_i}|< s$.  On the other aspect, since $F_{x_i}$ is also a subtree of $T$, we have $F_{x_i}=P_{l_i}$ for $l_i\le s$. Note that $s+1=|F|\ge l_1+l_2+1$, we have $l_i<s$. Thus $F_{x_i}$ must be a subtree $P_s$, however $F$ does not contain any $P_s$, a contradiction.
Therefore,  $d_{T_v}(v)=1$. Thus $v$ has a unique adjacent vertex $y$ in $T_v$ such that $T_v-vy$ has  a pendent subtree  of order $s$ that will be $P_s$ with root vertex $y$,  i.e., $T_v-vy=P_s\cup\{v\}$. It follows that  $T_v=P_{s+1}$. Since
$n-1=s k+t=(k-t)s+t(s+1)$, $r_{s+1}=t$ and $r_i=0$ for $i>s+1$, we have $(r_1+r_2+\cdots+r_{s+1})=(k-t)s+t(s+1)$. It means that $T-u=(k-t)*P_s\cup t*P_{s+1}$.
\end{proof}

Let $T_{l_1,l_2,\ldots,l_k}$ be a starlike tree with center vertex $u$ such that
\begin{equation}\label{r-eq-1}T_{l_1,l_2,\ldots,l_k}-u=\cup_{i=1}^k P_{l_i}\end{equation}
where $P_{l_i}$ is the path with  $l_i$ vertices. $T_{l_1,l_2,\ldots,l_k}$ is called balanced if  $|l_i-l_j|\le 1$ for $1\leq i,j \leq k$. Clearly, the starlike tree $T_{l_1,l_2,\ldots,l_k}$ described  in (a) and (b) of Lemma \ref{r-lem-1} are balanced. Thus Lemma \ref{r-lem-1} implies that  the edge division vector $\mathbf{r}(T)=(r_1,r_2,\ldots,r_{\lfloor\frac{n}{2}\rfloor})$ uniquely determine balanced starlike trees. Thus we have

\begin{corollary}
The balanced starlike trees are  DEDV-trees. Particularly, star $S_n$ and path $P_n$ are DEDV-trees.
\end{corollary}

However, the following result indicates that the starlike tree is not necessary to be balanced for a DEDV-tree. Let $(m*n)=(\overbrace{n,n,\ldots,n}^m)$ denote a sequence of $n$ of size $m$.

\begin{proposition}\label{pro-s-3}
Starlike tree $T_{l_1,l_2,l_3}$ is a  DEDV-tree.
\end{proposition}
\begin{proof}
A tree has three leaves if and only if it is a $T_{l_1,l_2,l_3}=(uv_1+P_{l_1})+(uv_2+P_{l_2})+(uv_3+P_{l_3})$, where $P_{l_i}$ is a path with $l_i$ vertices and $1\le l_1\le l_2\le l_3$. According to definition, we have
$$\mathbf{r}=\mathbf{r}(T_{l_1,l_2,l_3})=(l_1*3, (l_2-l_1)*2,(l_3-l_2)*1,0,\ldots,0).
$$
Thus a tree $T$ with $\mathbf{r}$ as above must have three leaves by Claim \ref{r-claim-1}(a), and so $T=(uv_1+P_{l_1'})+(uv_2+P_{l_2'})+(uv_3+P_{l_3'})$ where $P_{l_i'}$ is a path with $l_i'$ vertices and $1\le l_1'\le l_2'\le l_3'$. Hence $\mathbf{r}(T)=(l_1'*3, (l_2'-l_1')*2,(l_3'-l_2')*1,0,\cdots,0)=\mathbf{r}$, which leads to $l_i=l_i'$ and so $T\cong T_{l_1,l_2,l_3}$. Hence $T_{l_1,l_2,l_3}$ is a DEDV-tree.
\end{proof}

Let $\mathcal{T}_{n,k}$ be the set of all the starlike trees of the form $T_{l_1,l_2,\ldots,l_k}$ with order $n$. By the similar consideration of Proposition \ref{pro-s-3}, we have the following result.

\begin{lemma}\label{lem-3-2}
Let $T_{l_1,l_2,\ldots,l_k}, T_{l'_1,l'_2,\ldots,l'_k}\in \mathcal{T}_{n,k}$. Then $T_{l_1,l_2,\ldots,l_k}\approx T_{l'_1,l'_2,\ldots,l'_k}$ if and only if $T_{l_1,l_2,\ldots,l_k}\cong T_{l'_1,l'_2,\ldots,l'_k}$.
\end{lemma}

The starlike tree $T_{l_1,l_2,\ldots,l_k}$ described in (\ref{r-eq-1}) is called \emph{weak balanced} if $l_i+l_j\ge\max\{l_q\mid 1\le q\le k\}$ for $1\le i,j\le k$. In general, we have the following result.
\begin{theorem}\label{st-thm-1}
Let $k\ge 4$. The starlike tree $T_{l_1,l_2,\ldots,l_k}$ is a DEDV-tree if and only if it is weak balanced.
\end{theorem}
\begin{proof}
Let  $T=T_{l_1,l_2,\ldots,l_k}$ be a DEDV-tree described as in (\ref{r-eq-1}), and without loss of generality assume that $1\le l_1\le l_2\le\cdots\le l_k$.
If $T$ is not weak balanced, then $l_1+l_2< l_k$. There exists one edge $e=wz\in E(P_{l_k})$ such that $\mu(e)=|T_z|=l_1+l_2=|P_{l_1}|+|P_{l_2}|$.  Then we can get  $T'$  by branch-exchange of $T$ with $P_{l_1}\cup P_{l_2}$ and $T_z$. By Theorem \ref{thm-3-2-1}, we have $T'\approx T$ but $T'\not\cong T$ since $T'$ has two branching vertices, a contradiction.

Conversely, suppose that starlike tree  $T=T_{l_1,l_2,\ldots,l_k}$ is weak balanced, i.e., $l_1+l_2\geq l_k$ and $l_k< \lfloor\frac{n}{2}\rfloor$. According to definition, we have
$$\mathbf{r}=\mathbf{r}(T)=(l_1*k, (l_2-l_1)*(k-1),(l_3-l_2)*(k-2),\ldots,(l_k-l_{k-1})*1,0,\ldots,0)
$$
where $r_{l_k+1}(T)=\cdots=r_{\lfloor\frac{n}{2}\rfloor}(T)=0$.
Suppose that $T'$ is the tree with edge division vector $\mathbf{r}(T')=\mathbf{r}(T)$. Since $r_1=r_2=\cdots=r_{l_1}=k$, $T'$ contains exactly $k$ pendent paths $P_{l_1}$ by Lemma \ref{r-lem-1}, particularly $T'$ has $k$ leaves. First we will show that $T'$ is also a starlike tree. Since otherwise $T'$ has two branching vertices $u_1$ and $u_2$. Let $P'=u_1 u_1'\cdots u_2'u_2$ be the path in $T'$ connecting $u_1$ and $u_2$. Notice that $d_{T'}(u_1),d_{T'}(u_2) \ge 3$, we see that $T_{u_1}'(u_1u_1')$ contains at least two pendent paths $P_{l_1}$ and $P_{l_2}$, thus $|T_{u_1}'(u_1u_1')|\ge |P_{l_1}|+|P_{l_2}|+1=l_1+l_2+1>l_k$. Similarly, $|T_{u_2}'(u_2u_2')|\ge |P_{l_1}|+|P_{l_2}|+1=l_1+l_2+1>l_k$. Therefore,  $\mu_{T'}(e)\ge \min\{|T_{u_1}'(u_1u_1')|, |T_{u_2}'(u_2u_2')|\}>l_k$ for any edge $e\in E(P')$. It implies that $r_{i}\not=0$ for some $i>l_k$, which contradicts the assumption of $\mathbf{r}$. Thus $T'$ is a starlike tree such that $T'\in \mathcal{T}_{n,k}$ and  $T'\approx T$. It immediately follows $T'\cong T =T_{l_1,l_2,\ldots,l_k}$ by Lemma \ref{lem-3-2}.
\end{proof}

A \emph{double star} $S_{p,q}$ is the tree obtained from $K_2$ by attaching $p-1$ pendent vertices to one vertex and $q-1$ pendent vertices to the other vertex, where $p+q=n$.

\begin{lemma}\label{r-lem-2}
Let $\mathbf{r}(T)=(r_1,r_2,\ldots,r_{\lfloor\frac{n}{2}\rfloor})$ be the edge division vector of $T\in \mathcal{T}_n$. Suppose that $r_i=0$ if $i\not=1,p$, where $p<\lfloor\frac{n}{2}\rfloor$. We have\\
(a)  If $r_{p}=1$  then $T=S_{p,n-p}$.\\
(b)  If $r_{p}=t$ then $T$ has $t$ pendent $S_{p}$, each centre  of them  joins  a unique centre vertex of $T$, such a $T$ is called power star and denote by $S_{p}^t$.
\end{lemma}
\begin{proof}

(a) Since  $r_p=1$ and $r_i=0$ for $i\not=1,p$,  $T$ has only one non-pendent edge, say $e$, and thus $r_1=n-2$ by Claim \ref{r-claim-1}. Moreover, we have $\mu(e)=p$ due to  $r_{p}=1$. It follows that  $T=S_{p,n-p}$.

(b) Notice  $r_{p}=t\geq 2$, we may assume $T_u$ is  a pendent subtree of order $p$ with root $u$. If $T_u$ is not a star with centre $u$, then $d_{T_u}(u)<p-1$ and thus $T_u$ contains a path $P=uu'\cdots x'$, where $x'$ is a leaf of $T$. Then $1<\mu_{T}(uu')\leq p-1$, which contradicts $r_2=\cdots=r_{p-1}=0$. Therefore, $T$ contains $t$ copies of pendent star $S_p$ and each centre of them joins a unique centre vertex of $T$. Hence $T=S_{p}^t$.
\end{proof}

From Lemma \ref{r-lem-2} we see that  double star $S_{p,q}$ and power star $S_{p}^t$ are uniquely determined by their edge division vectors. Thus we have

\begin{corollary}
The double star $S_{p,q}$ and power star $S_{p}^t$ are DEDV-trees.
\end{corollary}

Let $T=P_{s_1}\bullet_uP_{s_2}+P_{uv}+P_{t_1}\bullet_vP_{t_2}$ be a tree on $n$ vertices with exactly two branching vertices $u$ and $v$ connecting by a path $P_{uv}=uu_1\cdots u_{k-1}v$ such that $T_u=T_u(uu_1)$ and $T_v=T_v(u_{k-1}v)$ are two starlike trees $T_u=T_{s_1,s_2}$ and $T_v=T_{t_1,t_2}$, respectively, where $s_1+s_2\leq t_1+t_2$, $s_1\le s_2$ and $t_1\le t_2$.

\begin{proposition}\label{pro-4-4}
Let $T=P_{s_1}\bullet_uP_{s_2}+P_{uv}+P_{t_1}\bullet_vP_{t_2}$, where $P_{uv}$ is a path of length $k$, $s_1+s_2\leq t_1+t_2$, $s_1\le s_2$ and $t_1\le t_2$. Then $T$ is a DEDV-tree if and only if the one of the following four conditions holds\\
(i) $s_1+s_2=t_1+t_2$,\\
(ii) $s_1+s_2>t_2$,\\
(iii) $s_1+s_2+k=t_2$ and $s_1+s_2> t_1$,\\
(iv) $s_1+s_2+k=t_2=t_1$.
\end{proposition}
\begin{proof}
We prove necessity by proving its inverse proposition. Suppose that $s_1+s_2\not=t_1+t_2$, we will show (ii), (iii) or (iv) holds. By  the way of contradiction, we may assume that
$s_1+s_2\le t_2$ and consider the following two situations.
{\flushleft\bf Case 1. } $s_1+s_2= t_2$;

If $s_1+s_2= t_2$, we can construct a tree $T'\approx T $ by branch-exchange the two balanced components $T_u-u=P_{s_1}\cup P_{s_2}$ and $P_{t_2}$. It is clear that   $T'=T_{t_1,s_1,s_2, t_2+k}$  is a starlike tree and so $T'\not\cong T$.  This means that  counterexamples can be found no matter how we choose the conditions under which (iii) and (iv) do not hold.
{\flushleft\bf Case 2. } $s_1+s_2< t_2$;

If $s_1+s_2+k\not=t_2$, we can get a tree $T'\approx T $ by branch-exchange the two balanced components $T_u-u=P_{s_1}\cup P_{s_2}$ and $P_{s_1+s_2}\subseteq P_{t_2}$ of $T$. It is clear that $T'$ is a tree with two branching vertices $u'$ and $v'$ connecting with a path $P_{u'v'}$ of $t_2-(s_1+s_2)+1$ vertices such that $T'_{u'}=T_{s_1,s_2}$ and $T'_{v'}=T_{t_1,s_1+s_2+k}$. Obviously, $T'\not\cong T$ because of $s_1+s_2+k\not=t_2$, a contradiction.

If $s_1+s_2+k=t_2= t_1+k$, it also means $t_2\neq t_1$, we can get a tree $T'\approx T $ by branch-exchange the two balanced components $T_u-u=P_{s_1}\cup P_{s_2}$ and $P_{t_1}$ of $T$. It is clear that $T'=T_{t_2,s_1,s_2, t_1+k}$ is a starlike tree and so $T'\not\cong T$, a contradiction.

If $s_1+s_2+k=t_2<t_1+k$ and $t_2\neq t_1$, we can get a tree $T'\approx T $ by branch-exchange the two balanced components $T_u-u=P_{s_1}\cup P_{s_2}$ and $P_{s_1+s_2}\subseteq P_{t_1}$ of $T$. It is clear that $T'$ is a tree with two branching vertices $u'$ and $v'$ connecting with a path $P_{u'v'}$ of $t_1-(s_1+s_2)+1$ vertices such that $T'_{u'}=T_{s_1,s_2}$ and $T'_{v'}=T_{s_1+s_2+k,t_2}$. Obviously, $T'\not\cong T$ because of $s_1+s_2+k\not=t_1$, a contradiction.

For the sufficiency, we may assume that $\alpha=s_1+s_2\le t_1+t_2=\beta$,  there exists a tree $T'\approx T $ with $(r_1,\ldots,r_{\lfloor\frac{n}{2}\rfloor})=\mathbf{r}(T)=\mathbf{r}(T')=(r_1',\ldots,r_{\lfloor\frac{n}{2}\rfloor}')$. Since $T $ has four leaves, we may assume  $P_{a_i}$ is the maximal pendent path of  $T'$ with $a_i$ vertices, where $a_1\le a_2\le a_3\le a_4$.

Suppose (i) holds, i.e., $\alpha=\beta$. Without loss of generality, assume that
$s_1\le t_1\le t_2\le s_2$.  It is easy to see that $r'_{s_1+s_2+1}=r_{s_1+s_2+1}=1$ if $k=1$ and $2$ otherwise, but in any case $T'$ has two pendent subtrees $T_{u'}'$ and $T_{v'}'$ such that $a_4<|T_{u'}'|=|T_{v'}'|=s_1+s_2+1\le\frac{n}{2}$. It implies that $|T_{u'}'|=a_{1}+a_{4}+1$ and $|T_{v'}'|=a_{2}+a_{3}+1$  due to $\alpha=\beta$. Let  $P_{u'v'}'$ be the path connecting $u'$ and $v'$, we have $T'=T_{u'}'+P_{u'v'}'+T_{v'}'$ and so $|P_{u'v'}'|=k$ by Claim \ref{r-claim-1} (d).
Moreover, we have $a_1=s_1$ by Claim \ref{r-claim-1} (c), and thus $a_4=s_2$.
Since $t_1+1<s_1+s_2$, we have $r_{s_1+1}=\cdots=r_{t_1}=3$ and $r_{t_1+1}=2$. Similarly, $r_{a_1+1}'=\cdots=r_{a_2}'=3$ and $r_{a_2+1}'=2$. Thus one can verify that  $t_1\ge a_2$. By symmetry, we get  $t_1= a_2$ and so $t_2=a_3$. Therefore,  $T'\cong T $.

In what follows we always assume that $\alpha<\beta$ because of (i).

Suppose (ii) holds, i.e., $t_2< s_1+s_2<t_1+t_2$.
First of all,  we claim that  $T'$ has exactly two branching vertices. Otherwise $T'=T_{a_1,a_2,a_3,a_4}$.  Since $r_{\alpha+1}\not=0$, we have $a_4\ge \alpha+1$ and so   $r_{i}=r_{i}'\not=0$ for $i=1,\ldots,a_4$. On the other hand, we have $r_i=0$ for $i=\max\{t_2,s_2\}+1,\ldots,\alpha$, and $\alpha< a_4$, a contradiction. Thus, we may assume that $T'$ has two branching vertices $u'$ and $v'$ connecting with a path  $P_{u'v'}'$. Then $T'=T_{u'}'+P_{u'v'}'+T_{v'}'$ with $a_{i_1}+a_{i_2}+1=|T_{u'}'|< |T_{v'}'|=a_{i_3}+a_{i_4}+1$, where $\{i_1,i_2,i_3,i_4\}=\{1,2,3,4\}$.  Clearly,  $t=\max\{s_2,t_2\}\ge a_4$ since $r_t\not=0$ and  $r_{t+1}=0$. If  $t >a_4$ then $r_{a_4+1}'=r_{a_4+1}\not=0$. Therefore, $|T_{u'}'|=a_{i_1}+a_{i_2}+1=a_4+1\le t<s_1+s_2<\frac{n}{2}$. This is impossible since $r_i'\not=0$ for $i\in [a_4+1,s_1+s_2]$ but $r_{t+1}=0$.  Thus $t=a_4$. Now by deleting $P_{t}$ from $T$ and $T'$, respectively, we get $\tilde{r_i}=r_i-1=r_i'-1=\tilde{r'_i}$ for $i=1,2,\ldots,c=\min\{s_2,t_2\}$.  It is easy to see that $\mathbf{r}(T_3)=(\tilde{r_1},\ldots,\tilde{r_c},0,\ldots,0)$ and $\mathbf{r}(T_3')=(\tilde{r'_1},\ldots,\tilde{r'_c},0,\ldots,0)$  are the edge division vectors of $T_3=T_{c_1,c_2,c_3}$ and $T_3'=T_{a_1,a_2,a_3}$, respectively, where  $\{c_1,c_2,c_3\}=\{s_1,s_2,t_1,t_2\}\setminus t$. By Proposition \ref{pro-s-3},  $\{c_1,c_2,c_3\}=\{a_1,a_2,a_3\}$  and so $\{s_1,s_2,t_1,t_2\}=\{a_1,a_2,a_3,a_4\}$. It implies that $T_{u'}'=P_{s_{1}}\bullet_{u'} P_{s_{2}}$  and $T_{v'}'=P_{t_{1}}\bullet_{v'} P_{t_{2}}$. It follows that $T\cong T'$.

Suppose (iii) holds, i.e., $s_1+s_2+k=t_2$ and $s_1+s_2> t_1$. Since $s_1+s_2+k=t_2<\lfloor\frac{n}{2}\rfloor$, we have $r_{t_2}'=r_{t_2}=2$ and $r_{t_2+1}'=r_{t_2+1}=0$. Then $T'$ has two pendent subtrees $T_{x_1'}'$ and $T_{x_2'}'$ of order $t_2$ and no any  of order $t_2+1$. Thus the root $x_i'$  of $T_{x_i'}'$  appends with  branching vertex $x_i$, i.e., $T_{x_i}'=T_{x_i'}'+x_i'x_i$ where $i=1,2$. First of all, both of $T_{x_1'}'$ and $T_{x_2'}'$ can not be  path since otherwise $r_{s_1+s_2}'\ge 2$ but $r_{s_1+s_2}=1$.  Moreover, $x_1=x_2=v'$ and one of $T_{x_1'}'$ and $T_{x_2'}'$ must be a path since otherwise $T'$ has at least five leaves. Thus  we may assume that $T_{x_2'}'=P_{t_2}$.  By deleting the $P_{t_2}$ from $T$ and $T'$, similar as the arguments in the proof of (ii) we get $\{a_1,a_2,a_3,a_4\}=\{s_1,s_2,t_1,t_2\}$. Let $P_{a_3}$ be the path, its root appends at $v'$. Then $a_3=n-2t_2-1=t_1$. It follows that $T'=T_{u'}'+P_{u'v'}'+T_{v'}'$, where $T_{u'}'=P_{s_1}\bullet_{u'} P_{s_2}$ and $T_{v'}'=P_{t_1}\bullet_{v'} P_{t_2}$, and thus $T\cong T'$.

Suppose (iv) holds. Let $s=s_1+s_2+k=t_2=t_1$, we have $r_s'=r_s=3$ and $r_i'=r_i=0$ for $i>s$. Then $T'$ has three pendent subtrees $T_{x_i'}'$ of order $s$,  its root $x_i'$ appends with  branching vertex $x_i$ for $i=1,2,3$, and no any  of order $s+1$. Since $T'$ has four leaves, $x_i$ can not distinct from each other. Thus we may assume that $x_1=x_2=x_3$ or $x_1=x_2\not=x_3$. Whichever happens, $T'$ contains the maximal pendent path $P_s$. By deleting the $P_{s}$ from $T$ and $T'$, similar as the arguments in the proof of (ii) we get $\{a_1,a_2,a_3,a_4\}=\{s_1,s_2,t_1,t_2\}$. It is clear that $T'\not=T_{a_1,a_2,a_3,a_4}$ because $r_i'=r_i=3$ for $i\in [s_1+s_2+1,t_2]$. Thus $T'=P_{s_1}\bullet_{u'} P_{s_2}+P_{u'v'}'+P_{t_1}\bullet_{v'} P_{t_2}\cong T$.

We complete this proof.
\end{proof}

Let $T$ be a tree on $n$ vertices with exactly two branching vertices $u$ and $v$ connecting by a path $P_{uv}=uu_1\cdots u_{k-1}v$ such that $T_u=T_u(uu_1)$ and $T_v=T_v(u_{k-1}v)$ are two starlike tree, where
\begin{equation}\label{st-eq-1}\left\{\begin{array}{ll}
T_u-u=\cup_{i=1}^{k_1} P_{s_i}, \mbox{ where $1\leq s_1\le s_2\le \cdots\le s_{k_1}$}\\
T_v-v=\cup_{i=1}^{k_2}P_{t_i}, \mbox{ where  $1\leq t_1\le t_2\le \cdots\le t_{k_2}$. }
 \end{array}\right.
\end{equation}
$T$ is called \emph{equivalent double starlike tree} if $s_1=\cdots=s_{k_1}=t_1=\cdots=t_{k_2}=s$, denoted by $DT_{s,k_1,k_2}$. Specially, $DT_{1,k_1,k_2}$ is called \emph{double broom graph}. By using  the above symbols we have the following result.

\begin{theorem}\label{D-thm-1}
Let $T=DT_{s,k_1,k_2}\in \mathcal{T}_n$ be a equivalent double starlike tree with two branching vertices $u$ and $v$ defined in Eq.(\ref{st-eq-1}). If $|k_1-k_2|\leq1$, then $T$ is DEDV-tree.
\end{theorem}
\begin{proof}
Let $\mathbf{r}(T)=(r_1,r_2,\ldots,r_{\lfloor\frac{n}{2}\rfloor})$ be the edge division vector of $T$. Without loss of generality, we assume that $k_1\leq k_2$. If $k_1=k_2-1$, according to definition, we have
\begin{equation}\label{qq-1}\left\{\begin{array}{ll}
r_j=k_1+k_2&j=1,\ldots,s\\
r_{j}=0& j=s+1,\ldots,sk_1\\
r_{j}=1& j=sk_1+1.\\
\end{array}\right.
\end{equation}
Assume that there exists $T'\approx T$, i.e.,  $\mathbf{r}(T')=(r_1',r_2',\ldots,r_{\lfloor\frac{n}{2}\rfloor}')=\mathbf{r}(T)$, in what follows we will show that  $T'\cong T$.  First we would verify the following Claims:\\
(i) $T'$ has $k_1+k_2$ leaves, each of them is included in a pendent path $P_s$;\\
(ii) $T'$ contains a pendent starlike tree $T'_{u'}$ with centre vertex $u'$ and $T'_{u'}-u'=k_1* P_{s}$.

In fact, (i) is obvious by  Lemma \ref{r-lem-1}. At last, we prove (ii). From (\ref{qq-1}) we know that $r'_{sk_1+1}=r_{sk_1+1}=1$. It implies that $T'$ has pendent subtree $T_{u'}'$ with root vertex $u'$ and $|T_{u'}'|=sk_1+1$. Since $r'_{j}=r_{j}=0$ for $j=s+1,\ldots, sk_1$ and (i) we know that any maximal pendent subtree of $T'$ properly included  in $T_{u'}'$ must be path $P_s$. Therefore, $T_{u'}'$ is a starlike tree with $|T_{u'}'|=sk_1+1$, i.e., $T'_{u'}-u'=k_1* P_{s}$. Recall that $k_2=k_1+1$. $T'$ must contain a pendent subtree $T'_{v'}$ with root vertex $v'$ and $T'_{v'}-v'=k_2* P_{s}$ since $r'_{sk_1+1}=r_{sk_1+1}=1$.
Let  $P_{u'v'}'$ be the path connecting $u'$ and $v'$, we have $T'=T'_{u'}+P_{u'v'}'+T'_{v'}$ and so $|P_{u'v'}'|=k$ by Claim \ref{r-claim-1} (d). Therefore,  $T'\cong T $.
If $k_1=k_2$, we get $T'_{u'}-u'=T'_{v'}-v'=k_1* P_{s}$. By the similar arguments mentioned in previous situation, we get $T'\cong T$.

We complete this proof.
\end{proof}

If $s=1$, such tree  $DT_{1,\lfloor\frac{n-k-1}{2}\rfloor,\lceil\frac{n-k-1}{2}\rceil}$ called the \emph{balanced double broom tree}.
At a special case of Theorem \ref{D-thm-1}, we have the following result.

\begin{corollary}\label{thm-3-4}
If $|k_1-k_2|\leq1$, then $DT_{1,k_1,k_2}$ is a DEDV-tree.
\end{corollary}

At the last of this section, we will give a method to construct DEDV-trees from some known DEDV-trees. We begin with some notions and symbols. Given a graph $G$ with vertex set $V=\{v_1,v_2,\ldots,v_n\}$ and a graph $H$ with root vertex $u$, the rooted product graph $G\diamond_u H$ is defined as the graph obtained from $G$ and $n$ copies of $H$ by identifying  vertex $v_i$ in $G$ with $u$ in the $i$-th copy of $H$. Let $P_s$ be the path of order $s$ with  one end point $u$ as its root vertex,  the rooted product graph $\bar{G}_s=G\diamond_u P_s$ is shown in Figure \ref{fig-root}. The corona product graph $G\circ H$ is defined as the graph obtained from $G$ and  $n$ copies of $H$ by joining the vertex $v_i$ of $G$ to every vertex in the $i$-th copy of $H$. If we take $H=sK_1$ ($s\geq 1$), then the corona product graph $\tilde{G}_s=G\circ sK_1$ is shown in Figure \ref{fig-root}. Particularly, if $G$ is taken as a tree $T\in \mathcal{T}_n$, we denote $\bar{T}_s=T\diamond_u P_s$ and $\tilde{T}_s=T\circ sK_1$. Let $\mathbf{r}(T)=(r_1,r_2,\ldots,r_{\lfloor\frac{n}{2}\rfloor})$ and $\mathbf{r}(\bar{T}_s)=(\bar{r}_1,\bar{r}_2,\ldots)$, $\mathbf{r}(\tilde{T}_s)=(\tilde{r}_1,\tilde{r}_2,\ldots)$. According to definition, one can simply verify that
\begin{equation}\label{gamma-eq-1}\bar{r}_i=\left\{\begin{array}{ll}
n& \mbox{ if $i\le s-1$}\\
r_{k}& \mbox{ if $i=ks$, where $k=1,2,\ldots$}\\
0& \mbox{ if $s<i\not=ks$}
\end{array}\right.
\end{equation}
\begin{equation}\label{gamma-eq-2}\tilde{r}_i=\left\{\begin{array}{ll}
ns& \mbox{ if $i=1$}\\
r_{k}& \mbox{ if $i=k(s+1)$, where $k=1,2,\ldots$}\\
0& \mbox{ if $1<i\not=k(s+1)$}
\end{array}\right.
\end{equation}
It is clear that $\mathbf{r}(\bar{T}_s)$ and  $\mathbf{r}(\tilde{T}_s)$ are  determined by $\mathbf{r}(T)$. Using the above symbols,  we can state   the following result.

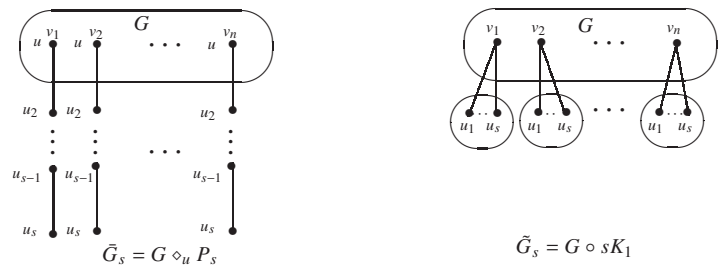
\begin{figure}[htb]
  \centering
\unitlength 1mm 
\linethickness{0.4pt}
\ifx\plotpoint\undefined\newsavebox{\plotpoint}\fi 
\begin{picture}(92.631,32.548)(0,0)
\put(16.794,27.577){\oval(33.588,9.546)[]}
\put(4.419,27.9){\circle*{1}}
\put(10.253,27.9){\circle*{1}}
\put(28.284,27.9){\circle*{1}}
\put(4.419,27.577){\line(0,-1){9.016}}
\put(4.419,11.667){\line(0,-1){9.192}}
\put(4.4,18.738){\circle*{1}}
\put(4.4,11.314){\circle*{1}}
\put(4.4,2.652){\circle*{1}}
\put(3.8,13.3){$\vdots$}
\put(17,26.7){$\cdots$}
\put(3.5,29){\tiny$v_1$}
\put(9.2,29){\tiny$v_2$}
\put(27,29){\tiny$v_n$}
\put(1.8,27.577){\tiny$u$}
\put(0.4,18.562){\tiny$u_2$}
\put(-1,10){\tiny$u_{s-1}$}
\put(.4,2.652){\tiny$u_s$}
\put(10.253,27.952){\line(0,-1){9.016}}
\put(10.253,12.042){\line(0,-1){9.192}}
\put(10.253,19.114){\circle*{1}}
\put(10.076,11.689){\circle*{1}}
\put(10.253,3.027){\circle*{1}}
\put(9.5,13.3){$\vdots$}
\put(7.2,27.5){\tiny$u$}
\put(6.2,18.5){\tiny$u_2$}
\put(6,10){\tiny$u_{s-1}$}
\put(6.2,2.6){\tiny$u_s$}
\put(28.284,27.952){\line(0,-1){9.016}}
\put(28.284,12.042){\line(0,-1){9.192}}
\put(28.284,19.114){\circle*{1}}
\put(28.107,11.689){\circle*{1}}
\put(28.284,3.027){\circle*{1}}
\put(27.5,13.3){$\vdots$}
\put(25,27.2){\tiny$u$}
\put(24,18.2){\tiny$u_2$}
\put(23,10){\tiny$u_{s-1}$}
\put(24,3.027){\tiny$u_s$}
\put(17,12.5){$\cdots$}
\put(15.203,29.7){\scriptsize$G$}
\put(75.837,27.775){\oval(33.588,9.546)[]}
\put(63.462,28.129){\circle*{1}}
\put(69.296,28.129){\circle*{1}}
\put(87.327,27.952){\circle*{1}}
\put(76,27){$\cdots$}
\put(62,29.5){\tiny$v_1$}
\put(68,29.5){\tiny$v_2$}
\put(86,29.5){\tiny$v_n$}
\put(74.246,29.5){\scriptsize$G$}
\put(59.751,18.738){\circle*{1}}
\put(63.463,18.738){\circle*{1}}
\multiput(63.286,28.107)(-.03344424,-.089184639){111}{\line(0,-1){.089184639}}
\put(63.286,28.284){\line(0,-1){9.899}}
\put(58.5,16.5){\tiny$u_1$}
\put(62,16.5){\tiny$u_s$}
\put(60,18){\tiny$\cdots$}
\put(61.43,17.501){\oval(8.309,7.071)[]}
\put(68.943,18.919){\circle*{1}}
\put(72.657,18.919){\circle*{1}}
\put(67.5,16.5){\tiny$u_1$}
\put(71.5,16.5){\tiny$u_s$}
\put(69,18){\tiny$\cdots$}
\put(76,18){$\cdots$}
\put(70.622,17.677){\oval(8.309,7.071)[]}
\put(85.03,18.92){\circle*{1}}
\put(88.744,18.92){\circle*{1}}
\put(83.5,16.5){\tiny$u_1$}
\put(87.5,16.5){\tiny$u_s$}
\put(86,18){\tiny$\cdots$}
\put(86.709,17.677){\oval(8.309,7.071)[]}
\put(69.12,19.092){\line(0,1){9.192}}
\multiput(69.12,28.284)(.033587572,-.097227182){100}{\line(0,-1){.097227182}}
\multiput(85.03,18.738)(.03330575,.135785){69}{\line(0,1){.135785}}
\multiput(87.151,28.284)(.03335409,-.18678292){53}{\line(0,-1){.18678292}}
\put(11,-1){\scriptsize$\bar{G}_s=G\diamond_u P_s$}
\put(66,0){\scriptsize$\tilde{G}_s=G\circ sK_1$}
\end{picture}
  \caption{The rooted product graph $G\diamond_u P_s$ and corona product graph $G\circ sK_1$}\label{fig-root}
\end{figure}

\begin{theorem}\label{PP-thm-1}
Let $T$ be a DEDV-tree of order $n$. Then we have\\
(a) $\bar{T}_s=T\diamond_uP_s$ ($s\geq 1$) is DEDV-tree.\\
(b) $\tilde{T}_s=T\circ sK_1$ ($s\geq 1$) is DEDV-tree.
\end{theorem}
\begin{proof}
First we prove (a). Suppose that there is a tree $H\approx\bar{T}_s$ and $H_x$ is any pendent subtree  of $H$ with root $x$. By Lemma \ref{r-lem-1}, we see that $H$ has exactly $n$   pendent paths $P_{s-1}$ due to $\bar{r}_{i}=n$ for $i=1,2,\ldots,s-1$, thus $H_x\cong P_i$ for $i=|H_x|<s$ and is included in  $P_{s-1}$. Now by deleting these $n$  pendent paths $P_{s-1}$ from $H$, we get a subtree $T'$ from  $H$. Suppose that  $\mathbf{r}(T')=(r_1',r_2',\ldots,r_{\lfloor\frac{n}{2}\rfloor}')$. In what follows we only need to show that each vertex of $T'$ joins one end point of path $P_{s-1}$ in $H$ and thus $\mathbf{r}(T')=\mathbf{r}(T)$ according to (\ref{gamma-eq-1}). Consequently,  $T'\cong T$ due to $T$ is a DEDV-tree.

From (\ref{gamma-eq-1}) we know that any pendent subtree of $H$ with order less than $s$ is a path by Lemma \ref{r-lem-1}. Note that $\bar{r}_{s}=r_{1}\not=0$, $H$ has  pendent subtree $H_x$ with $|H_x|=s$. Again by Lemma \ref{r-lem-1}, the root $x$ is pendent vertex of $T'$ and so $H_x=P_s$. Let $H_y$ be any  pendent subtree of $H$ with order $|H_y|=ks\ge s$, where $k\ge 1$, and assume that each vertex of $V(T')\cap V(H_y)$  joins one end point of $P_{s-1}$. Now if $i=|H_z|>ks$ then $i=k's$ for some $k'>k$ because $\bar{r}_i=0$ if $s\nmid i$. Let $z_1,\ldots,z_t$ be  adjacent vertices of $z$ in $H_z$ and $H_{z_i}$ be the component of $H_z-z_iz$ containing $z_i$ where $1\le i\le t$. By induction hypothesis, each vertex of $V(T')\cap V(H_{z_i})$ joins one end point of $P_{s-1}$ and hence $|H_{z_i}|=s|V(T')\cap V(H_{z_i})|$. Note that $|V(T')\cap V(H_z)|=|V(T')\cap V(H_{z_1})|+\cdots+|V(T')\cap V(H_{z_t})|+1$, where $z\in V(T')$ contributes $1$, and we have $s|V(T')\cap V(H_z)|=s|V(T')\cap V(H_{z_1})|+\cdots+s|V(T')\cap V(H_{z_t})|+s=|H_{z_1}|+\cdots+|H_{z_t}|+s$. Since $|H_z|=k's$ has the form of $s| V(T')\cap V(H_z)|$, we claim that $z$ joins exactly one end point of $P_{s-1}$. Therefore, each vertex of $T'$ joins one end point of $P_{s-1}$ by induction. By considering the edge division vector of $H$ we have
\begin{equation}\label{gamma-eq-3}\bar{r}_i=\left\{\begin{array}{ll}
n& \mbox{ if $i\le s-1$}\\
r_{k}'& \mbox{ if $i=ks$}\\
0& \mbox{ if $s<i\not=ks$.}
\end{array}\right.
\end{equation}
From (\ref{gamma-eq-1}) and (\ref{gamma-eq-3}) we see that $r_k=r_k'$ for $k\ge 1$, and so $\mathbf{r}(T)=\mathbf{r}(T')$.

As similar as (a), one can verify (b).
\end{proof}

\begin{example}\label{TT-exa-2}
According to Theorem \ref{PP-thm-1}, $\bar{(P_n)}_s=P_n\diamond_u P_s$ and $\tilde{(P_n)}_s=P_n\circ sK_1$ are DEDV-trees since $P_n$ is DEDV-tree.
\end{example}

\section{Construction for  EDV-equivalent trees and DEDV-trees}

In this section, we will use the branch-exchange transformation  to construct EDV-equivalent trees, especially we give all the  EDV-equivalent trees of vertices  no more than $10$ and consequently the corresponding DEDV-trees are also determined.

First of all, note that a tree of order less than $7$ has at most two branching vertices, we get the following result by Theorem \ref{st-thm-1} and Proposition \ref{pro-4-4}.

\begin{proposition}\label{pro-3-1-0}
There is no any non-isomorphic EDV-equivalent trees of  order less than $7$, and equivalently any tree $T$ of order $n$ is a  DEDV-tree if $n<7$.
\end{proposition}

Except for the individual tree, all the trees of order $7\le n\le 8$ have at most two branching vertices, we can select the trees that do not satisfy Theorem \ref{st-thm-1} and Proposition \ref{pro-4-4}, which are listed in the Figure \ref{fig-3}, fortunately one can verify that they are just all families of non-isomorphic EDV-equivalent trees of order $7\le n\le 8$. We summarize them in Proposition \ref{pro-3-1} and Proposition \ref{pro-3-2}.

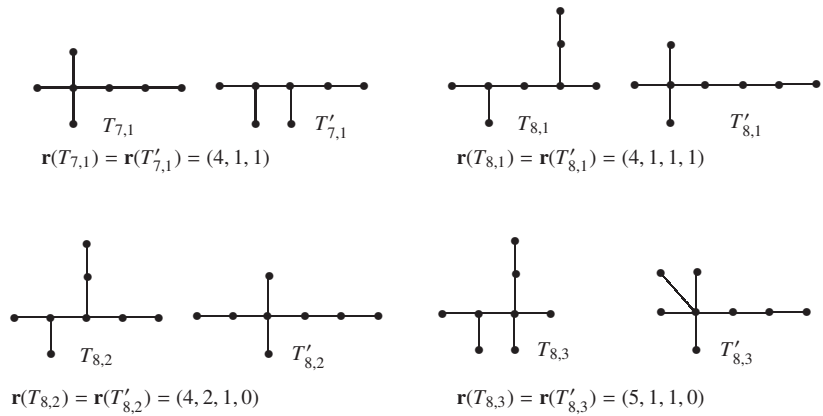
\begin{figure}[h]
\centering
\unitlength .8mm 
\linethickness{0.4pt}
\ifx\plotpoint\undefined\newsavebox{\plotpoint}\fi 
\begin{picture}(134.239,70.898)(0,0)
\put(4.302,57.65){\circle*{1.2}}
\put(10.293,57.65){\circle*{1.2}}
\put(16.284,57.65){\circle*{1.2}}
\put(22.276,57.65){\circle*{1.2}}
\put(28.267,57.65){\circle*{1.2}}
\put(10.293,63.649){\circle*{1.2}}
\put(10.293,51.561){\circle*{1.2}}
\put(4.091,57.658){\line(1,0){24.281}}
\put(10.293,63.859){\line(0,-1){12.298}}
\put(15.128,50.3){\scriptsize$T_{7,1}$}
\put(34.644,57.859){\circle*{1.2}}
\put(40.635,57.859){\circle*{1.2}}
\put(46.336,57.859){\circle*{1.2}}
\put(52.619,57.859){\circle*{1.2}}
\put(58.61,57.859){\circle*{1.2}}
\put(40.635,51.597){\circle*{1.2}}
\put(34.433,57.859){\line(1,0){24.281}}
\put(40.531,57.955){\line(0,-1){6.202}}
\put(46.418,58.06){\line(0,-1){6.517}}
\put(46.535,51.597){\circle*{1.2}}
\put(50.412,50.597){\scriptsize$T'_{7,1}$}
\put(5,45){\scriptsize$\mathbf{r}(T_{7,1})=\mathbf{r}(T'_{7,1})=(4,1,1)$}
\put(71.64,20.041){\circle*{1.2}}
\put(77.631,20.041){\circle*{1.2}}
\put(83.622,20.041){\circle*{1.2}}
\put(89.614,20.041){\circle*{1.2}}
\put(107.973,20.25){\circle*{1.2}}
\put(113.674,20.25){\circle*{1.2}}
\put(119.957,20.25){\circle*{1.2}}
\put(125.948,20.25){\circle*{1.2}}
\put(113.756,20.451){\line(0,-1){6.517}}
\put(113.873,13.988){\circle*{1.2}}
\put(87.315,12.988){\scriptsize$T_{8,3}$}
\put(117.75,12.988){\scriptsize$T'_{8,3}$}
\put(74,5){\scriptsize$\mathbf{r}(T_{8,3})=\mathbf{r}(T'_{8,3})=(5,1,1,0)$}
\put(83.627,31.975){\line(0,-1){12.298}}
\put(83.681,32.218){\circle*{1.2}}
\put(83.83,26.718){\circle*{1.2}}
\put(77.735,13.934){\circle*{1.2}}
\put(77.681,20.123){\line(0,-1){6.095}}
\put(113.803,26.474){\line(0,-1){6.202}}
\put(114.006,27.015){\circle*{1.2}}
\put(125.844,20.272){\line(1,0){6.243}}
\put(132.141,20.326){\circle*{1.2}}
\put(71.438,20.124){\line(1,0){18.135}}
\put(83.681,13.935){\circle*{1.2}}
\put(83.478,20.124){\line(0,-1){6.243}}
\put(125.844,20.124){\line(-1,0){18.135}}
\put(107.911,26.867){\circle*{1.2}}
\multiput(107.708,26.962)(.03359887,-.03779661){177}{\line(0,-1){.03779661}}
\put(73.238,57.923){\circle*{1.2}}
\put(79.229,57.923){\circle*{1.2}}
\put(85.22,57.923){\circle*{1.2}}
\put(91.212,57.923){\circle*{1.2}}
\put(97.203,57.923){\circle*{1.2}}
\put(73.027,57.931){\line(1,0){24.281}}
\put(84.064,50.573){\scriptsize$T_{8,1}$}
\put(103.58,58.132){\circle*{1.2}}
\put(109.571,58.132){\circle*{1.2}}
\put(115.272,58.132){\circle*{1.2}}
\put(121.555,58.132){\circle*{1.2}}
\put(127.546,58.132){\circle*{1.2}}
\put(103.369,58.132){\line(1,0){24.281}}
\put(119.348,50.87){\scriptsize$T'_{8,1}$}
\put(74,45){\scriptsize$\mathbf{r}(T_{8,1})=\mathbf{r}(T'_{8,1})=(4,1,1,1)$}
\put(79.333,51.816){\circle*{1.2}}
\put(79.279,58.005){\line(0,-1){6.095}}
\put(127.442,58.154){\line(1,0){6.243}}
\put(133.739,58.208){\circle*{1.2}}
\put(91.171,70.155){\line(0,-1){12.298}}
\put(91.225,70.398){\circle*{1.2}}
\put(91.374,64.898){\circle*{1.2}}
\put(109.392,58.131){\line(0,-1){6.517}}
\put(109.509,51.668){\circle*{1.2}}
\put(109.439,64.154){\line(0,-1){6.202}}
\put(109.642,64.695){\circle*{1.2}}
\put(.5,19.405){\circle*{1.2}}
\put(6.491,19.405){\circle*{1.2}}
\put(12.482,19.405){\circle*{1.2}}
\put(18.474,19.405){\circle*{1.2}}
\put(24.465,19.405){\circle*{1.2}}
\put(.289,19.413){\line(1,0){24.281}}
\put(11.326,12.055){\scriptsize$T_{8,2}$}
\put(30.842,19.614){\circle*{1.2}}
\put(36.833,19.614){\circle*{1.2}}
\put(42.534,19.614){\circle*{1.2}}
\put(48.817,19.614){\circle*{1.2}}
\put(54.808,19.614){\circle*{1.2}}
\put(30.631,19.614){\line(1,0){24.281}}
\put(42.616,19.815){\line(0,-1){6.517}}
\put(42.733,13.352){\circle*{1.2}}
\put(46.61,12.352){\scriptsize$T'_{8,2}$}
\put(0,5){\scriptsize$\mathbf{r}(T_{8,2})=\mathbf{r}(T'_{8,2})=(4,2,1,0)$}
\put(12.487,31.339){\line(0,-1){12.298}}
\put(12.541,31.582){\circle*{1.2}}
\put(12.69,26.082){\circle*{1.2}}
\put(6.595,13.298){\circle*{1.2}}
\put(6.541,19.487){\line(0,-1){6.095}}
\put(42.663,25.838){\line(0,-1){6.202}}
\put(42.866,26.379){\circle*{1.2}}
\put(54.704,19.636){\line(1,0){6.243}}
\put(61.001,19.69){\circle*{1.2}}
\end{picture}
  \caption{All the pairs of non-isomorphic EDV-equivalent trees on $7,8$ vertices}\label{fig-3}
\end{figure}

\begin{proposition}\label{pro-3-1}
There is exactly one pair of non-isomorphic EDV-equivalent trees on $7$ vertices, which are labelled as $T_{7,1}$ and $T_{7,1}'$ shown in Figure \ref{fig-3}, where $\mathbf{r}(T_{7,1})=\mathbf{r}(T_{7,1}')=(4,1,1)$.
\end{proposition}
Proposition \ref{pro-3-1} can also restate as that except of $T_{7,1}$ and $T_{7,1}'$ all the trees on $7$ vertices are DEDV-trees.
\begin{proposition}\label{pro-3-2}
There are exactly three pairs of non-isomorphic EDV-equivalent trees on $8$ vertices, which are labelled as $(T_{8,1}, T'_{8,1})$, $(T_{8,2}, T'_{8,2})$ and $(T_{8,3}, T'_{8,3})$ shown in Figure \ref{fig-3}, where $\mathbf{r}(T_{8,1})=\mathbf{r}(T'_{8,1})=(4,1,1,1)$, $\mathbf{r}(T_{8,2})=\mathbf{r}(T'_{8,2})=(4,2,1,0)$ and $\mathbf{r}(T_{8,3})=\mathbf{r}(T'_{8,3})=(5,1,1,0)$.
\end{proposition}
Proposition \ref{pro-3-2} can also restate as that except of $(T_{8,1}, T'_{8,1})$, $(T_{8,2}, T'_{8,2})$ and $(T_{8,3}, T'_{8,3})$ all the trees on $8$ vertices are DEDV-trees.

The pairs of non-isomorphic EDV-equivalent trees
on 7 and 8 vertices described in Proposition \ref{pro-3-1} and Proposition \ref{pro-3-2}. Similarly,
we can exhaust all the non-isomorphic EDV-equivalent trees on 9 vertices, which correspond to $11$ distinct  edge division vectors: $(4,1,1,2)$, $(4,1,2,1)$, $(4,2,1,1)$, $(4,2,2,0)$, $(5,1,1,1)$, $(5,1,2,0)$, $(5,2,0,1)$, $(5,2,1,0)$, $(6,1,1,0)$,  $(6,1,0,1)$, $(6,0,1,1)$.

\begin{proposition}
There exist only $11$  classes of the non-isomorphic EDV-equivalent trees on $9$ vertices, which are listed in the Figure \ref{fig-8}.
\end{proposition}

\begin{figure}[h]
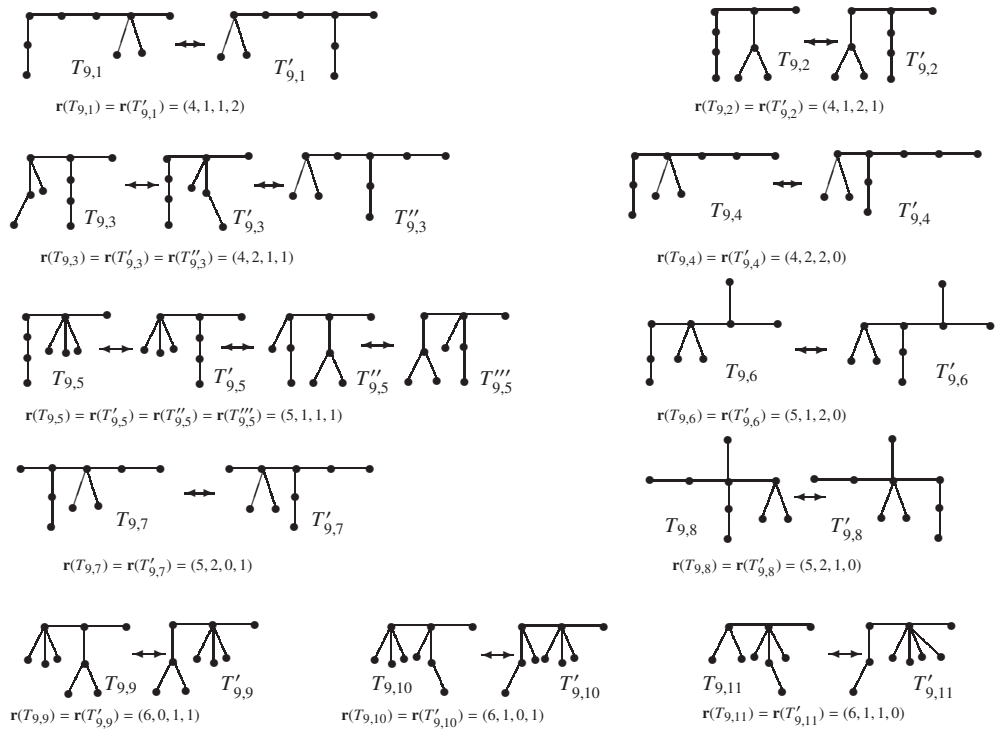

  \centering
\unitlength 1mm 
\linethickness{0.4pt}
\ifx\plotpoint\undefined\newsavebox{\plotpoint}\fi 

  \caption{All the pairs of non-isomorphic EDV-equivalent trees on $9$ vertices}\label{fig-8}
\end{figure}

From Figure \ref{fig-8} we see that there are four non-isomorphic trees corresponding to one edge division vectors $(5,1,1,1)$. Additionally, we mention that except of the 11 classes trees list in Figure \ref{fig-8} all the trees on $9$ vertices are DEDV-trees.

\begin{remark}
There are $106$ connected trees on $10$ vertices,  $59$ of which can be divided into $25$ classes based on edge division vector that are EDV-equivalent trees, and the remaining $47$ trees are DEDV-trees (see Appendix A).
\end{remark}

Let $u$ and $v$ be two vertices of $T$, recall that call $u$ and $v$  similar if there is  an automorphism $\alpha$ of $T$  that contains  the transposition $(u\ v)$ (i.e., $\alpha(u)=v$ and $\alpha(v)=u$), and not similar otherwise. In the following result, we  give a simple method of constructing non-isomorphic EDV-equivalent trees that is a special case of the branch-exchange.

\begin{theorem}\label{thm-3-1}
Let $T$ and  $T'=T-ux+vx$ be assumed as in Lemma \ref{lem-2-2} and $n_u(uv)-n_v(uv)=|T_x(ux)|$. Let $T^*=T-ux-T_x(ux)$. If $u$ and $v$ are not similar vertices of $T^*$ then $T\approx T'$ but $T\not\cong T'$.
\end{theorem}
\begin{proof}
Since  $n_u(uv)-n_v(uv)=|T_x(ux)|$, we have $T\approx T'$ from Lemma \ref{lem-2-2}(2). Also note that $T=T^*+xu+T_x(ux)$ and $T'=T^*+xv+T_x(ux)$, we see that the two components $T^*_u(uv)$ and $T^*_v(uv)$  have the same number of vertices, i.e., $|T^*_u(uv)|=n_u(uv)-|T_x(ux)|=n_v(uv)=|T^*_v(uv)|$. If $T\cong T'$ then there exists an isomorphism $\alpha$ such that $\alpha(T)=T'$. It implies that $\alpha$ must fix the edge $uv$ (i.e., $\alpha$ maps the edge $uv$ of $T$ to the edge $uv$ of $T'$). Therefore,  $\alpha(T_u(uv))=T'_v(uv)$ and $\alpha(T_v(uv))=T'_u(uv)$ because of $|T_u(uv)|=|T_v(uv)|+|T_x(ux)|=|T'_v(uv)|>|T_v(uv)|=|T_u(uv)|-|T_x(ux)|=|T'_u(uv)|$.
Since $T_v(uv)=T^*_v(uv)$ and $T'_u(uv)=T^*_u(uv)$, we have $\alpha(T^*_v(uv))=T^*_u(uv)$. This proves that $u$ and $v $ are similar vertices of $T^*$, a contradiction. Therefore, $T\not\cong T'$.
\end{proof}

According to Theorem \ref{thm-3-1}, we can give a method of constructing non-isomorphic EDV-equivalent trees. Given any tree $T^*$ with one specified edge $uv$, we construct the trees $T$ and $T'$ shown in Figure \ref{fig-0}. If $T^*$ satisfies the following conditions:\\
(a) $u$ is not similar with $v$ in $T^*$,\\
(b) $|T^*_u(uv)|= |T^*_v(uv)|$,\\
we have $T\approx T'$ but $T\not\cong T'$. It leads the following conclusion.

\begin{proposition}
There exist infinite pairs of EDV-equivalent trees that are not isomorphic such that they contain any given tree $T^*$ as their subtree.
\end{proposition}
\begin{remark}
It is an interesting problem  to find the necessary and sufficient condition of the transformation for a tree $T$  such that it can produce all the trees from $T$ that are EDV-equivalent with $T$ but not isomorphic to it.
\end{remark}

We examine trees having the same edge division vector in Figure \ref{fig-3} and Figure \ref{fig-8} and find that each one can be obtained from another by using the  branch-exchange transformation. We now propose a problem.

\begin{problem}
Let $T$ be a tree and $\mathcal{B}(T)$, the set of trees that can be obtained from $T$ by using branch-exchange transformation. Then $T$ is DEDV-tree if and only if $\mathcal{B}(T)$ does not contain any tree except of $T$ itself.
\end{problem}

\section{Topological indices on trees}
The first topological index, named Wiener index, was introduced in \cite{Wiener} and is defined as  $W(G)=\sum_{\{u,v\}\in V(G)}d(u,v)$. It was already known to Wiener that on the class of  trees Wiener index  can be represented by a function of $\mu(e)$ as follows:
\begin{equation}\nonumber\begin{array}{ll}
W(T)&=\sum_{e=xy\in E(T)}n_x(e)n_y(e)=\sum_{e=xy\in E(T)}n_x(e)(n-n_x(e))\\
&=\sum_{e\in E(T)}\mu(e)(n-\mu(e)).
\end{array}
\end{equation}

Let $\mathbf{r}=(r_1,r_2,\ldots,r_{\lfloor\frac{n}{2}\rfloor})$ be the edge division vector of $T$ and $f(x)=x(n-x)$. As in the proof of Theorem 8 in \cite{Vuki}, $W(T)$ can be further simplified as
\begin{equation}\label{WF-eq-22}W(T)=\sum_{e\in E(T)}\mu(e)(n-\mu(e))=\sum_{1\le i\le \lfloor\frac{n}{2}\rfloor}r_if(i).\end{equation}
The authors in \cite{Vuki} introduced the notions bellow.

\begin{definition}\label{def-3-1}
Let $F:\mathcal{T}_n\rightarrow \mathbb{R}$ be a topological index and let $f:\mathbb{N}\rightarrow \mathbb{R}$ be a real function defined for positive integers. The  topological index $F$ is an edge additive eccentric topological index if it holds that
$F(T)=\sum_{e\in E(T)}f(\mu(e))=\sum_{1\le i\le \lfloor\frac{n}{2}\rfloor}r_if(i)$.
Function $f$ is called the edge contribution function of index $F$.
\end{definition}

In \cite{Vuki}, the authors showed that several well known topological indices on the class of trees can be also  represented by some functions of $\mu(e)$ described as Eq. (\ref{WF-eq-22}), and they are edge additive eccentric topological indices.
They are Wiener index \cite{Wiener}, modified Wiener indices \cite{Gutman-1}, variable Wiener indices \cite{Vukic-1} and Steiner $k$-Wiener index \cite{Li-1}. Recently, Song and Huang et al. \cite{Song} extended the above conclusions to the more topological indices: hyper-Wiener index \cite{Rand}, Wiener-Hosoya index \cite{Randic}, degree distance \cite{Klein}, Gutman index \cite{Gutman} and  second atom-bond connectivity index \cite{Graovac}. All these topological indices and their edge contribution functions are summarized in the Table \ref{tab-1}.

\begin{table}[htbp]
\caption{Some edge additive eccentric topological indices \label{tab-1}}
\centering
\scriptsize
\begin{tabularx}{430pt}{c|c|c}
\toprule
Indices & Definition & Edge contribution function    \\
\midrule
\makecell[t]{  Wiener  index } &
\makecell[t]{ $W(T)=\sum_{\{u,v\}\in V(T)}d(u,v)$} &
\makecell[t]{ $f(x)=x(n-x)$} \\
\midrule
\makecell[t]{ Modified  Wiener  indices } &
\makecell[t]{ $^{\lambda}W(T)=\sum_{\{u,v\}\in V(T)}d^{\lambda}(u,v)$} &
\makecell[t]{ $f(x)=x^\lambda(n-x)^\lambda$} \\
\midrule
\makecell[t]{ Variable  Wiener  indices } &
\makecell[t]{ $_{\lambda}W(T)=\frac{1}{2}\sum_{e=uv\in E(T)}(n^{\lambda}-{n_u(e)}^{\lambda}-{n_v(e)}^{\lambda})$} &
\makecell[t]{$f(x)=n^\lambda-x^\lambda-(n-x)^\lambda$ } \\
\midrule
\makecell[t]{ Steiner $k$-Wiener  index } &
\makecell[t]{ $SW_k(T)=\sum_{e=uv\in E(T)}\sum_{i=1}^{k-1}\binom{n_u(e)}{i}\binom{n_v(e)}{k-i}$} &
\makecell[t]{$f(x)=\binom{n}{k}-\binom{x}{k}-\binom{n-x}{k}$}   \\
\midrule
\makecell[t]{ hyper-Wiener  index } &
\makecell[t]{ $WW(T)=\sum_{e=uv\in E(T)}(\frac{1}{2}n_u(e)n_v(e)+\frac{1}{2}n_u(e)^2n_v(e)^2)$} &
\makecell[t]{$f(x)=\frac{1}{2}x(n-x)+\frac{1}{2}x^2(n-x)^2$}   \\
\midrule
\makecell[t]{ Wiener-Hosoya  index } &
\makecell[t]{ $h(T)=\sum_{e=uv\in E(T)}[n_u(e)n_v(e)+(n_u(e)-1)(n_v(e)-1)]$} &
\makecell[t]{$f(x)=x(n-x)+(x-1)(n-x-1)$}  \\
\midrule
\makecell[t]{ degree  distance } &
\makecell[t]{ $D'(T)=\sum_{e=uv\in E(T)}(4n_u(e)n_v(e)-n)$} &
\makecell[t]{$f(x)=4x(n-x)-n$}  \\
\midrule
\makecell[t]{ Gutman  index } &
\makecell[t]{ $Gut(T)=\sum_{e=uv\in E(T)}[4n_u(e)n_v(e)-(2n-1)]$} &
\makecell[t]{$f(x)=4x(n-x)-(2n-1)$}  \\
\midrule
\makecell[t]{ second  atom-bond \\ connectivity  index } &
\makecell[t]{ $ABC_2(T)=\sum_{e=uv\in E(T)}\sqrt{\frac{n-2}{n_u(e)(n-n_u(e))}}$} &
\makecell[t]{$f(x)=\sqrt{n-2}x^{-\frac{1}{2}}(n-x)^{-\frac{1}{2}}$}  \\
\bottomrule
\end{tabularx}
\end{table}

Instead of dealing with the extremal problem for individual topological index one by one, we can use the concept of edge additive eccentric topological index to unify the problem of determining  the extreme topological index in certain classes of trees. In the literature \cite{Song} the authors gave and summarized all the results for the extremal  topological indices involved in Table \ref{tab-1}. Here, at the last of the paper, we turn to consider whether trees with different structures can have the same topological indices? Based on the relation of $\langle \mathcal{T}_n,\preceq\rangle$ and the notion of  edge additive eccentric topological index, the following results give us a way to solve this problem.

\begin{theorem}\label{thm-6-1}
Let $F:\mathcal{T}_n\longrightarrow \mathbb{R}$ be an edge additive eccentric topological index and let $T,T' \in \mathcal{T}_n$. If $T\approx T'$, then $F(T)=F(T')$.
\end{theorem}
\begin{proof}
Since $F$ is an edge additive eccentric topological index, we know that it is defined by
$$F(T)=\sum_{e\in E(T)}f(\mu(e)),$$
where $f$ is its edge contribution function. Let $\mathbf{r}=(r_1,r_2,\ldots,r_{\lfloor\frac{n}{2}\rfloor})$ be the edge division vector of $T$ and $\mathbf{r'}=(r'_1,r'_2,\ldots,r'_{\lfloor\frac{n}{2}\rfloor})$ be the edge division vector of $T'$. By the definition of topological index we have
\[\left\{\begin{array}{ll}
F(T)=\sum_{1\le i\le \lfloor\frac{n}{2}\rfloor}r_if(i),\\
F(T')=\sum_{1\le i\le \lfloor\frac{n}{2}\rfloor}r'_if(i).\\
\end{array}\right.
\]
Note that $\mathbf{r}=\mathbf{r'}$ since $T\approx T'$. Therefore, we have $F(T)=F(T')$.
\end{proof}

Theorem \ref{thm-6-1} shows that non-isomorphic EDV-equivalent trees have the same topological index value, which does not depend on individual form of topological index.
\begin{table}[htb]
\footnotesize
\caption{\footnotesize The values of topological index of all non-isomorphic EDV-equivalent trees of order $7,8,$ and $9$}\label{tab-3}
\centering
\begin{tabular*}{15cm}{p{20pt}p{90pt}p{90pt}p{50pt}p{50pt}p{50pt}}
\hline
 $n$ & trees & edge division vector & $W(\cdot)$ & $h(\cdot)$ & $Gut(\cdot)$ \\\hline
7 & $T_{7,1},T'_{7,1}$ & $(4,1,1)$     &46     &56 &106\\
8 & $T_{8,1},T'_{8,1}$& $(4,1,1,1)$    &71     &93 &179\\
8 & $T_{8,2},T'_{8,2}$& $(4,2,1,0)$    &67     &85 &163 \\
8 & $T_{8,3},T'_{8,3}$& $(5,1,1,0)$    &62     &75 &143\\
9 & $T_{9,1},T'_{9,1}$& $(4,1,1,2)$    &104     &144 &280 \\
9 & $T_{9,2},T'_{9,2}$& $(4,1,2,1)$    &102     &140 &272\\
9 & $T_{9,3},T'_{9,3},T''_{9,3}$& $(4,2,1,1)$    &98      &132 &256 \\
9 & $T_{9,4},T'_{9,4}$& $(4,2,2,0)$    &96     &128 &248\\
9 & $T_{9,5},T'_{9,5},T''_{9,5},T'''_{9,5}$& $(5,1,1,1)$    &92  &120 &232    \\
9 & $T_{9,6},T'_{9,6}$& $(5,1,2,0)$    &90     &116 &224\\
9 & $T_{9,7},T'_{9,7}$& $(5,2,0,1)$    &88      &112 &216 \\
9 & $T_{9,8},T'_{9,8}$& $(5,2,1,0)$    &86     &108 &208\\
9 & $T_{9,9},T'_{9,9}$& $(6,0,1,1)$    &86      &108 &208 \\
9 & $T_{9,10},T'_{9,10}$& $(6,1,0,1)$    &82     &100 &192\\
9 & $T_{9,11},T'_{9,11}$& $(6,1,1,0)$    &80    &96 &184  \\\hline
\end{tabular*}
\end{table}

Notice that the pairs of non-isomorphic EDV-equivalent trees on $7$, $8$ and $9$ vertices along with their edge division vectors are described in Figure \ref{fig-3} and Figure \ref{fig-8} (other trees that are not depicted are DEDV-trees), from which we give  Table \ref{tab-3} as an example of application  that list the values of various topological indices  of these non-isomorphic EDV-equivalent trees. From Table \ref{tab-3} we see that pairs of these non-isomorphic EDV-equivalent trees have the same value of topological index although the value depending on individual form of topological index would be varied.

\begin{remark}
From Table \ref{tab-3} we know that the EDV-equivalent trees have the same value of topological index, and sometimes the trees with different edge division vectors can also have the same topological index, for example the trees $T_{9,8},T'_{9,8}$ and $T_{9,9},T'_{9,9}$ given in Table \ref{tab-3}.
\end{remark}

\begin{table}[htb]
\footnotesize
\caption{\footnotesize Fractions of DEDV trees and EDV-equivalent trees}\label{tab-4}
\centering
\begin{tabular*}{15cm}{p{15pt}p{40pt}p{60pt}p{100pt}p{60pt}p{90pt}}
\hline
 $n$ & $\#$ trees & $\#$ DEDV trees& $\#$ EDV-equivalent trees & DEDV trees & EDV-equivalent trees \\\hline
2  & 1   & 1     &0      &1         &0      \\
3  & 1   & 1     &0      &1         &0          \\
4  & 2   & 2     &0      &1         &0          \\
5  & 3   & 3     &0      &1         &0         \\
6  & 6   & 6     &0      &1         &0         \\
7  & 11  & 9     &2      &0.8181    &0.1818     \\
8  & 23  & 17    &6      &0.7391    &0.2609     \\
9  & 47  & 22    &25     &0.4681    &0.5319     \\
10 & 106 & 47    &59     &0.4434    &0.5566    \\\hline
\end{tabular*}
\end{table}

To determine the all pairs of EDV-equivalent trees  we first of all had to generate the all connected trees by computer and then determine their edge division vectors. These would have to be stored and then compared. For example, the trees on $7$, $8$ and $9$ vertices are described in Section 5, and there are $106$ trees on $10$ vertices, $59$ of which are divided into $25$ classes based on edge division vector that are EDV-equivalent trees, and the remaining $47$ trees are DEDV-trees (see Appendix A). The results are in Table \ref{tab-4}, where we give the fractions of EDV-equivalent trees and DEDV-trees. Notice that for $n\leq 6$ there are no EDV-equivalent trees, all of them are DEDV-trees. An interesting result from the table is that the fraction of EDV-equivalent trees is nondecreasing for small $n$. If this tendency continues, almost all trees will be the EDV-equivalent pairs in the table. Indeed, the fraction of trees that are DEDV-trees tends to zero as $n$ tends to infinity. The conclusion may be that the present data give some indication that, the fraction of non-isomorphic EDV-equivalent pairs tends to one as $n$ tends to infinity.

\appendix
\section{All DEDV-trees on $10$ vertices}

\begin{figure}[htb]
\centering
\unitlength .9mm 
\linethickness{0.4pt}
\ifx\plotpoint\undefined\newsavebox{\plotpoint}\fi 

\end{figure}

\end{document}